\documentclass{amsart}
\usepackage{dsfont, amssymb, mathtools, IEEEtrantools, subcaption}
\usepackage[utf8]{inputenc}
\usepackage[english]{babel}
\usepackage[T1]{fontenc}
\usepackage{lmodern}
\usepackage{calc,changepage}
\usepackage{csquotes}
\usepackage{hyperref}
\usepackage{enumitem, siunitx}

\usepackage{xcolor}
\definecolor{wb}{RGB}{51,153,255}
\usepackage{tikz}
\usetikzlibrary{calc, patterns, decorations.pathreplacing}
\usepackage{tikz-cd}

\usepackage[parfill]{parskip}
\numberwithin{equation}{subsection}
\usepackage[style=numeric, sorting=nyt]{biblatex}
\newcommand{\defeq}{\vcentcolon=}
\newcommand{\eqdef}{=\vcentcolon}

\makeatletter
\def\moverlay{\mathpalette\mov@rlay}
\def\mov@rlay#1#2{\leavevmode\vtop{%
   \baselineskip\z@skip \lineskiplimit-\maxdimen
   \ialign{\hfil$\m@th#1##$\hfil\cr#2\crcr}}}
\newcommand{\charfusion}[3][\mathord]{
    #1{\ifx#1\mathop\vphantom{#2}\fi
        \mathpalette\mov@rlay{#2\cr#3}
      }
    \ifx#1\mathop\expandafter\displaylimits\fi}
\makeatother

\newcommand{\cupdot}{\charfusion[\mathbin]{\cup}{\cdot}}

\newcommand{\longhookrightarrow}{\lhook\joinrel\longrightarrow}

\makeatletter
\newcommand{\bigperp}{%
  \mathop{\mathpalette\bigp@rp\relax}%
  \displaylimits
}

\newcommand{\bigp@rp}[2]{%
  \vcenter{
    \m@th\hbox{\scalebox{\ifx#1\displaystyle2.1\else1.5\fi}{$#1\perp$}}
  }%
}
\makeatother

\newtheoremstyle{definitions}
 	{\topsep}
	{\topsep}
	{}
	{}
	{\bfseries}
	{:}
	{.5em}
	{}
  
\newtheoremstyle{lemmata}
	{\topsep}
	{\topsep}
	{\itshape} 
	{}
	{\bfseries}
	{:}
	{.5em}
	{}

\theoremstyle{lemmata}
\newtheorem{Theorem}[subsection]{Theorem}
\newtheorem*{Theorem-nn}{Theorem}
\newtheorem{Lemma}[subsection]{Lemma}
\newtheorem{Corollary}[subsection]{Corollary}
\newtheorem{Corollary-Definition}[subsection]{Corollary/Definition}
\newtheorem{Proposition}[subsection]{Proposition}
\newtheorem{manualtheoreminner}{}
\newenvironment{manualtheorem}[1]{%
  \renewcommand\themanualtheoreminner{#1}%
  \manualtheoreminner
}{\endmanualtheoreminner}

\theoremstyle{definitions}
\newtheorem{Definition}[subsection]{Definition}
\newtheorem{Remark}[subsection]{Remark}
\newtheorem{Remarks}[subsection]{Remarks}
\newtheorem*{Remark-nn}{Remark}
\newtheorem*{Remarks-nn}{Remarks}

\newtheorem{Example}[subsection]{Example}

\DeclareMathOperator{\rk}{rk}
\DeclareMathOperator{\GL}{GL}

\DeclareMathOperator{\PGL}{PGL}

\DeclareMathOperator{\Quot}{Quot}

\DeclareMathOperator{\Gr}{Gr}
\DeclareMathOperator{\Characteristic}{char}

\DeclareMathOperator{\Spec}{Spec}

\DeclareMathOperator{\sgn}{sgn}
\DeclareMathOperator{\Disc}{Disc}
\DeclareMathOperator{\Maps}{Maps}

\newcommand{\ldot}{\mathrel{\;.\;}}

\begin{document}

\title[Distributive Properties]{Distributive properties of division points and discriminants of Drinfeld modules}
\author{Ernst-Ulrich Gekeler}
\address{FR 6.1 Mathematik, Universität des Saarlandes, Postfach 15 11 50 D-66041 Saarbrücken}
\email{gekeler@math.uni-sb.de}
\date{\today}

\begin{abstract}
	We present a new notion of distribution and derived distribution of rank \(r \in \mathds{N}\) for a global function field \(K\) with a distinguished place \(\infty\).
	It allows to describe the relations between division points, isogenies, and discriminants both for a fixed Drinfeld module of rank \(r\) for the above data, or for the
	corresponding modular forms.
	
	We introduce and study three basic distributions with values in \(\mathds{Q}\), in the group \(\mu(\overline{K})\) of roots of unity in the algebraic closure
	\(\overline{K}\) of \(K\), and in the group \(U^{(1)}(C_{\infty})\) of \(1\)-units of the completed algebraic closure \(C_{\infty}\) of \(K_{\infty}\), respectively.
	
	There result product formulas for division points and discriminants that encompass known results (e.g. analogues of Wallis' formula for \((2\pi \imath)^{2}\) in the
	rank-\(1\) case, of Jacobi's formula \(\Delta = (2\pi \imath)^{12} q \prod (1-q^{n})^{24}\) in the rank-\(2\) case, and similar boundary expansions for \(r > 2\))
	and several new ones: the definition of a canonical discriminant for the most general case of Drinfeld modules and the description of the sizes of division and
	discriminant forms.
	
	In the now classical case where \((K, \infty) = (\mathds{F}_{q}(T), \infty)\) and \(r = 1\), \(2\) or \(3\), we give explicit values for the logarithms of such forms.
\end{abstract} 

\maketitle

\section{Introduction and Notation} \label{Section.Introduction-and-Notation}

\subsection{} This paper originates in the author's attempt to find a common framework for work of Carlitz \cite{Carlitz35}, Hayes \cite{Hayes74}, \cite{Hayes85}, Galovich-Rosen \cite{GalovichRosen81}, and the author
\cite{Gekeler85}, \cite{Gekeler84}, \cite{Gekeler86} Ch. IV and VI about division points of Drinfeld modules of rank \(1\) and of higher rank. Since long time it appeared that the product formulas
for periods of Drinfeld modules of rank \(1\) given in \cite{Gekeler86} Ch. IV, for discriminants in the higher rank case (\cite{Gekeler86} Ch. VI, \cite{gekeler2023drinfeld}, see also \cite{Basson17}), and similar expressions
for the respective division points, should have a common source.

Ideally, one would hope for a representation of the discriminant as a (suitably regularized as in \cite{Gekeler86} Ch. IV) product over the elements of the period lattice,
and such that the division points are described by partial products.

That this is essentially true is the content of Theorem \ref{Theorem.Main-Theorem} along with its corollaries.

\subsection{} \label{Subsection.First-introduction-of-notation} %
Let us introduce some notation. Throughout, \(K\) is a global function field with \(\mathds{F}_{q}\) as its field of constants, where \(q\) is a power of the prime number
\(p\), and \enquote{\(\infty\)} is a fixed place of \(K\), of degree \(d_{\infty}\) over \(\mathds{F}_{q}\). We let \(A \subset K\) be the Dedekind ring of elements of \(K\)
regular away from \(\infty\). Equivalently, \(K\) is the function field of a smooth projective geometrically connected curve \(\mathcal{C}\) over \(\mathds{F}_{q}\), and
\(A = \mathcal{O}_{\mathcal{C}}(\mathcal{C} \smallsetminus \{\infty\})\). Such rings are called \textbf{Drinfeld coefficient rings} for short. As \(q\) is fixed throughout,
we omit it from notation and write \(\mathds{F}\) for \(\mathds{F}_{q}\). Let further \(K_{\infty}\) be the completion of \(K\) at \(\infty\), and \(\pi \in K\) a 
uniformizer at \(\infty\). We normalize the absolute value \(\lvert \ldot \rvert = \lvert \ldot \rvert_{\infty}\) by 
\begin{equation}\stepcounter{subsubsection}%
	\lvert \pi \rvert^{-1} = q_{\infty} \defeq q^{d_{\infty}},
\end{equation}
and let \(C_{\infty}\) be the completed algebraic closure of \(K_{\infty}\) with respect to \(\lvert \ldot \rvert\).

\subsubsection{}\stepcounter{equation}\label{Subsubsection.Simple-example-C-projective-line-infty-place-at-infinity}%
The most simple example is where \(\mathcal{C}\) is the projective line and \enquote{\(\infty\)} the usual place at infinity, so \(K = \mathds{F}(T)\) with an indeterminate
\(T\), \(A = \mathds{F}[T]\), and \(K_{\infty} = \mathds{F}((T))^{-1}\). For \(x \in C_{\infty}\), we let
\begin{equation}
	\log x = \log_{q} \lvert x \rvert \quad (x \neq 0) \qquad \text{and} \qquad \log 0 = {-}\infty.
\end{equation}
If \(0 \neq x \in A\), then \(\log x\) agrees with \(\deg x = \dim_{\mathds{F}} A/(x)\).

\subsection{} An \(A\)-\textbf{lattice in} \(C_{\infty}\) is a finitely generated (hence projective of some rank \(r\)) \(A\)-submodule \(\Lambda\) of \(C_{\infty}\) which
is \textbf{discrete} in the sense that it has finite intersection with each ball of finite diameter in \(C_{\infty}\).

Fix \(r \in \mathds{N} = \{1,2,3,\dots\}\) and put \(V \defeq K^{r}\), \(V_{\infty} \defeq K_{\infty}^{r}\). An \(A\)-\textbf{lattice in} \(V\) is a finitely generated
\(A\)-submodule \(Y\) of \(V\) of full rank \(r\). Let
\begin{equation}
	\Psi = \Psi^{r} \defeq \{ \boldsymbol{\omega} = (\omega_{1}, \dots, \omega_{r}) \in C_{\infty}^{r} \mid \omega_{1}, \dots, \omega_{r} \text{ \(K_{\infty}\)-linearly independent} \}
\end{equation}
be the \textbf{Drinfeld space} of dimension \(r\). To each \(\boldsymbol{\omega} \in \Psi\) there corresponds an embedding
\begin{equation}
	\begin{split}
		i_{\boldsymbol{\omega}} \colon V			&\longhookrightarrow C_{\infty}. \\
		\boldsymbol{v} = (v_{1}, \dots, v_{r})	&\longmapsto \boldsymbol{v} \boldsymbol{\omega} \defeq \sum_{1 \leq i \leq r} v_{i}\omega_{i}.
	\end{split}	
\end{equation}
It has the property: Given a lattice \(Y\) in \(V\) and \(\boldsymbol{\omega} \in \Psi\), the set 
\(\Lambda = Y_{\boldsymbol{\omega}} \defeq i_{\boldsymbol{\omega}}(Y)\) is an \(A\)-lattice in \(C_{\infty}\). Vice versa, given a lattice \(\Lambda \subset C_{\infty}\),
the choice of a \(K\)-basis of \(K\Lambda\) yields an \(\boldsymbol{\omega} \in \Psi\) and \(Y \subset V\) such that \(\Lambda = Y_{\boldsymbol{\omega}}\). We
note that
\begin{equation}
	\Omega = \Omega^{r} \defeq C_{\infty}^{\times} \backslash \Psi
\end{equation}
is the classical Drinfeld symmetric space as defined in \cite{Drinfeld74}.

\subsection{} To each \(\Lambda\) as above, we let
\begin{equation}
	\begin{split}
		e^{\Lambda} \colon C_{\infty}		&\longrightarrow C_{\infty} \\
							z				&\longmapsto z \sideset{}{^{\prime}} \prod_{\lambda \in \Lambda} (1 - z/\lambda)	
	\end{split}	
\end{equation}
be its exponential function. Here and in the sequel we use the 

\subsection*{Convention} \(\sideset{}{^{\prime}} \prod (\cdots)\) resp. \( \sideset{}{^{\prime}} \sum (\cdots) \) is the product resp. sum over the non-zero elements of the 
respective index set.

The main properties of \(e^{\Lambda}\) are easily established: The product converges and defines an \(\mathds{F}\)-linear map onto \(C_{\infty}\) with kernel \(\Lambda\).
Furthermore, for each \(0 \neq a \in A\), there exists a polynomial \(\phi_{a}^{\Lambda} \in C_{\infty}[X]\) of shape
\begin{equation}\label{Eq.Members-of-Drinfeld-A-module}
	\phi_{a}^{\Lambda}(X) = aX + \sum_{1 \leq i \leq r \deg a} {}_{a}\ell_{i} X^{q^{i}}
\end{equation}
such that the diagram
\begin{equation}\label{Eq.Commutative-diagram-Drinfeld-A-module}
	\begin{tikzcd}
		0 \ar[r]	& \Lambda \ar[d, "a"] \ar[r]	 	& C_{\infty} \ar[r, "e^{\Lambda}"] \ar[d, "a"]	& C_{\infty} \ar[r] \ar[d, "\phi_{a}^{\Lambda}"] 	& 0 \\
		0 \ar[r]	& \Lambda \ar[r]								& C_{\infty} \ar[r, "e^{\Lambda}"]									& C_{\infty} \ar[r]												& 0
	\end{tikzcd}
\end{equation}
is commutative. Its coefficients \({}_{a}\ell_{i} = {}_{a}\ell_{i}(\Lambda) = {}_{a}\ell_{i}^{Y}(\boldsymbol{\omega})\) if \(\Lambda = Y_{\boldsymbol{\omega}}\) depend on
\(\Lambda\) or on \(\boldsymbol{\omega} \in \Psi\), and the \textbf{discriminant}
\begin{equation}
	\Delta_{a}(\Lambda) = \Delta_{a}^{Y}(\boldsymbol{\omega}) \defeq {}_{a} \ell_{r \deg a}(\Lambda)
\end{equation}
doesn't vanish. It is well-known that the collection \(\{ \phi_{a}^{\Lambda} \mid a \in A \}\) (where \(\phi_{0}^{\Lambda} = 0\)) defines a Drinfeld \(A\)-module 
\(\phi^{\Lambda}\) over \(C_{\infty}\), which establishes a bijective correspondence (see, e.g., \cite{Drinfeld74}, \cite{Goss96}, \cite{Rosen02})
\begin{equation}
	\{ \text{\(A\)-lattices in \(C_{\infty}\) of rank \(r\)} \} \longleftrightarrow \{ \text{Drinfeld \(A\)-modules over \(C_{\infty}\) of rank \(r\)} \}.	
\end{equation}
We note that for \(0 \neq c \in C_{\infty}\)
\begin{equation}
	{}_{a}\ell_{i}(c \Lambda) = c^{1-q^{i}} {}_{a}\ell_{i}(\Lambda)
\end{equation}
holds. In fact, the \({}_{a}\ell_{i}\) are modular forms for the group \(\Gamma_{Y} = \GL(Y)\) of weight \(q^{i} - 1\) and type \(0\) (\cite{BassonBreuerPink-tA}, \cite{gekeler2023drinfeld}).

\subsection{}\label{Subsection.Assumptions-on-lattice-and-fixed-element}%
Assume that \(\Lambda = Y_{\boldsymbol{\omega}}\) and \(0 \neq a \in A\) as before. By \eqref{Eq.Commutative-diagram-Drinfeld-A-module} and the properties of 
\(e^{\Lambda}\), the
\begin{equation}
	d_{\boldsymbol{u}}^{Y}(\boldsymbol{\omega}) \defeq e^{Y_{\boldsymbol{\omega}}}(\mathbf{u} \boldsymbol{\omega})
\end{equation}
with \(\mathbf{u} \in a^{-1}Y \smallsetminus Y\) are the non-trivial zeroes of \(\phi_{a}^{\Lambda}(X)\), i.e., the \(a\)-\textbf{division points} of \(\phi^{\Lambda}\), and
depend only on the class of \(\mathbf{u}\) in \(a^{-1}Y/Y\). Hence we may write
\begin{equation}
	\phi_{a}^{Y_{\boldsymbol{\omega}}}(X) = aX \sideset{}{^{\prime}} \prod_{\mathbf{u} \in a^{-1}Y/Y} (1 - d_{\boldsymbol{u}}^{Y}(\boldsymbol{\omega})^{-1}X).
\end{equation}
In particular,
\begin{equation}\label{Eq.Discriminant-as-function-of-bold-omega}
	\Delta_{a}^{Y} = a \Big( \sideset{}{^{\prime}} \prod_{\mathbf{u} \in a^{-1}Y/Y} d_{\mathbf{u}}^{Y} \Big)^{-1} = a \Delta_{(a)}^{Y}
\end{equation}
as functions of \(\boldsymbol{\omega}\). That is, \(\Delta_{(a)}^{Y}\) is \(\Delta_{a}^{Y}\) deprived of the factor \(a\).

\subsection{}\label{Subsection.Lattice-rank}%
Given two lattices \(\Lambda \subset \Lambda'\) of rank \(r\) in \(C_{\infty}\) (we call these a \textbf{lattice pair} of rank \(r\) in \(C_{\infty}\), and also write
\(\Lambda' | \Lambda\) for this data), the index \([\Lambda' : \Lambda]\) is finite, say, \([\Lambda' : \Lambda] = q^{d}\). In analogy with 
\eqref{Eq.Members-of-Drinfeld-A-module} and \eqref{Eq.Commutative-diagram-Drinfeld-A-module}, there exists an \(\mathds{F}\)-linear polynomial \(\varphi\) (that is,
a polynomial where only \(1\), \(q\), \(q^{2}\), \dots appear as exponents) with linear coefficient 1 (i.e., its derivative \(\varphi'\) equals \(1\))
and of degree \(q^{d}\) such that
\begin{equation}
	e^{\Lambda'} = \varphi \circ e^{\Lambda}.
\end{equation}
If \(\phi = \phi^{\Lambda}\) and \(\phi' = \phi^{\Lambda'}\) are the associated Drinfeld modules, then \(\varphi\) describes an isogeny from \(\phi\) to \(\phi'\). More
generally, if \(\phi\), \(\phi'\) are arbitrary Drinfeld \(A\)-modules of rank \(r\), where \(\phi = \phi^{\Lambda}\), \(\phi' = \phi^{\Lambda'}\), and
\(\varphi \colon \phi \to \phi'\) is an isogeny with \(\varphi' = c \in C_{\infty}\), then \(c \neq 0\), \(\Lambda \subset c^{-1}\Lambda'\), and replacing
\(\Lambda'\) with \(c^{-1}\Lambda'\) yields a lattice pair \(\Lambda' | \Lambda = Y_{\boldsymbol{\omega}}' | Y_{\boldsymbol{\omega}}\), which means that \(\varphi' = 1\).
Hence the collection of lattice pairs \(Y'|Y\) in \(V\) controls the collection of all possible isogenies of Drinfeld \(A\)-modules of rank \(r\). In the above situation,
we define
\begin{equation}\label{Eq.Discriminant-of-Phi}
	\Delta(\Lambda' | \Lambda) = \Delta^{Y'|Y}(\boldsymbol{\omega}) \defeq \text{the leading coefficient of \(\varphi\)},
\end{equation}
and call it the \textbf{discriminant of} \(\varphi\), or \textbf{of} \(\Lambda' | \Lambda\). Hence \(\Delta_{(a)}\) as defined in 
\eqref{Eq.Discriminant-as-function-of-bold-omega} equals \(\Delta(a^{-1}\Lambda | \Lambda) = \Delta^{a^{-1}Y|Y}(\boldsymbol{\omega})\). Generalizing 
\eqref{Eq.Discriminant-as-function-of-bold-omega}, we have
\begin{equation}
	\Delta^{Y'|Y}(\boldsymbol{\omega}) = \Big( \sideset{}{^{\prime}} \prod_{\mathbf{u} \in Y'/Y} d_{\mathbf{u}}^{Y}(\boldsymbol{\omega}) \Big)^{-1}.
\end{equation}

\subsection{} The family of all \(d_{\mathbf{u}}^{Y}\) and \(\Delta^{Y'|Y}\), where \(\mathbf{u}\) runs through \(V = K^{r}\) and \(Y'|Y\) through the set of lattice
pairs in \(V\), has some formal properties which are close to defining a distribution on the set
\begin{equation} \label{Eq.Definition-of-frak-G}
	\mathfrak{Y} = \{ (\mathbf{u}, Y) \mid \mathbf{u} \in V, Y \text{ an \(A\)-lattice in } V\}
\end{equation}
(see Section \ref{Section.Distributions-and-Derived-Distributions} for definitions). Our Theorem \ref{Theorem.Main-Theorem} along with its corollaries may be stated in
simplified form as follows.

\begin{manualtheorem}{Main Theorem}\begin{enumerate}[label=\(\mathrm{(\roman*)}\)]
	\item Given a lattice \(\Lambda\) of rank \(r\) in \(C_{\infty}\), there exists a canonical discriminant \(\Delta(\Lambda)\) such that for each \(0 \neq a \in A\),
	\begin{equation}
		\Delta_{a}(\Lambda) = \sgn(a) \Delta(\Lambda)^{(q^{r \deg a} - 1)/w_{r}}
	\end{equation}
	holds. Here \(w_{r} = q_{\infty}^{r} - 1\), and \(\sgn(a)\) is a \((q_{\infty} - 1)\)-th root of unity (see \ref{Subsection.Choices-for-1-units}). If 
	\(\Lambda = Y_{\boldsymbol{\omega}}\) with \(\boldsymbol{\omega}\) varying through \(\Psi\), then \(\Delta^{Y}(\boldsymbol{\omega}) = \Delta(Y_{\boldsymbol{\omega}})\)
	defines a modular form of weight \(w_{r}\) and type \(0\) for \(\Gamma_{Y} = \GL(Y)\).
	\item Let \(\mathcal{O}(\Psi)^{*}\) be the multiplicative group of invertible holomorphic functions on the Drinfeld space \(\Psi = \Psi^{r}\). Then
	\begin{equation}
		\begin{split}
			F \colon \mathfrak{Y}		&\longrightarrow \mathcal{O}(\Psi)^{*} \\
					(\boldsymbol{u},Y)	&\longmapsto \begin{cases} (d_{\mathbf{u}}^{Y})^{w_{r}} \Delta^{Y}	& (\mathbf{u} \notin Y) \\ \Delta^{Y}	&(\mathbf{u} \in Y) \end{cases}
		\end{split}
	\end{equation}
	is a distribution on \(\mathfrak{Y}\).
	\item Both sorts of functions, \(d_{\mathbf{u}}^{Y}(\boldsymbol{\omega})\) and \(\Delta^{Y}(\boldsymbol{\omega})\), may be evaluated through simple product formulas
	in terms of \(Y\), \(\mathbf{u}\), and \(\boldsymbol{\omega}\) (see Theorem \ref{Theorem.Main-Theorem}).
\end{enumerate}
\end{manualtheorem}

As a by-product, we get the sizes \(\lvert \Delta(\Lambda' | \Lambda) \rvert\), \(\lvert \Delta(\Lambda) \rvert\), \(\lvert d_{\mathbf{u}}^{Y}(\boldsymbol{\omega})\rvert\)
of the involved functions (see Corollary \ref{Corollary.Absolute-values-of-discriminants} and Section \ref{Section.The-case-A-IF-T}).

\subsection{} We start in Section \ref{Section.Distributions-and-Derived-Distributions} with defining the domain \(\mathfrak{Y}\) and the notions of distributions and derived
distributions on it. In Section \ref{Section.Basic-properties-of-the-functions}, the basic relations between the functions \(d_{\mathbf{u}}^{Y}\) and \(\Delta^{Y'|Y}\)
are investigated. These are such that they form a derived distribution on \(\mathfrak{Y}\) with values in the group \(\mathcal{O}(\Psi)^{*}\), in fact, with invertible
modular forms as values. As a consequence of the relations, we describe in Theorem \ref{Theorem.Action-of-Hecke-operator-on-Delta-form-for-lattice-pair} the multiplicative
action of Hecke correspondences on the discriminant forms \(\Delta^{Y'|Y}\). This is analogous with (but in view of the unbounded rank \(r\) more general than) the
results of Gilles Robert \cite{Robert90}, \cite{Robert90-1} about the classical elliptic discriminant.

In view of the decomposition of the multiplicative group (see \eqref{Eq.Decomposition-of-C-infty-star})
\[
	C_{\infty}^{*} = \pi^{\mathds{Q}} \times \mu(C_{\infty}) \times U^{(1)}(C_{\infty}),
\]
where \(\mu(C_{\infty})\) and \(U^{(1)}(C_{\infty})\) are the roots of unity and the \(1\)-units in \(C_{\infty}\), respectively, we next study certain distributions
with values in \(\mathds{Q}\), in \(U^{(1)}(C_{\infty})\), and in \(\mu(C_{\infty})\), in Sections \ref{Section.Z-functions-and-lattices}, \ref{Section.1-units}, and
\ref{Section.Roots-of-unity}. The \(\mu(C_{\infty})\)-valued distributions, treated in Section \ref{Section.Roots-of-unity}, behave differently and are delicate to handle.
For the \(\mathds{Q}\)-valued distributions, we introduce the \(Z\)-function of a lattice \(\Lambda\), an invariant analogous with the zeta function of a variety over a finite
field.

This enables us to show in Section \ref{Section.Product-expansions} our main result Theorem \ref{Theorem.Main-Theorem} (essentially item (iii) of the Main Theorem announced
above) and some of its consequences, among which the definition of the canonical discriminant \(\Delta\). We should note here that such a discriminant has been defined
in \cite{BassonBreuerPink-tA} and \cite{gekeler2023drinfeld}, but in each case only up to roots of unity. Here, due to the analysis in Section \ref{Section.Roots-of-unity}, we get the precise value of that root
of unity, and thereby the \textbf{canonical} \(\Delta(\Lambda) = \Delta^{Y}(\boldsymbol{\omega})\).

We conclude in Section \ref{Section.The-case-A-IF-T} with an application to the situation of \ref{Subsubsection.Simple-example-C-projective-line-infty-place-at-infinity},
where \(A = \mathds{F}[T]\). This connects the present results with existing ones (cases \(r=1\) and \(2\)), and produces new ones about sizes of modular forms (\(r = 2\) 
and \(3\)).

\subsection*{Notation}

\begin{itemize}
	\item \(\mathds{F} = \mathds{F}_{q}\) the finite field with \(q\) elements, of characteristic \(p\)
	\item \(K\) a global function field with \(\mathds{F}\) as field of constants
	\item \(\infty\) a fixed place of \(K\), of degree \(d_{\infty}\) over \(\mathds{F}\)
	\item \(A\) the Dedekind subring of \(K\) of elements regular off \(\infty\)
	\item \(K_{\infty}\) the completion of \(K\) at \(\infty\), with a uniformizer \(\pi \in K\)
	\item \(C_{\infty}\) the completed algebraic closure of \(K_{\infty}\)
	\item \(\lvert \ldot \rvert\) the absolute value on \(C_{\infty}\) with value group \(q^{\mathds{Q}}\), normalized by 
	\(\lvert \pi \rvert^{-1} = q_{\infty} = q^{d_{\infty}}\)
	\item \(\log x = \log_{q} \lvert x \rvert\) (\(0 \neq x \in C_{\infty}\)), \(\log 0 = {-}\infty\)
	\item \(w \defeq q_{\infty} - 1\), \(w_{r} \defeq q_{\infty}^{r} - 1\)
	\item \(r \in \mathds{N} = \{1,2,3,\dots \}\) a fixed natural number, the \textbf{rank} of our situation (usually omitted from notation)
	\item \(V = K^{r}\), \(V_{\infty} = K_{\infty}^{r}\), \(Y\), \(Y'\) \(A\)-lattices in \(V\)
	\item \(\mathfrak{Y} = \{ (\mathbf{u}, Y) \mid \mathbf{u} \in V, Y \subseteq V \text{ a lattice} \}\)
	\item \(\Psi = \Psi^{r}\), \(\Omega = C_{\infty}^{*} \backslash \Psi\) the Drinfeld spaces with groups \(\mathcal{O}(\Psi)^{*}\) resp. \(\mathcal{O}(\Omega)^{*}\) of
	invertible holomorphic functions
	\item \(\Lambda\), \(\Lambda'\) \(A\)-lattices of rank \(r\) in \(C_{\infty}\)
	\item \(\Lambda_{N} = \{ \lambda \in \Lambda \mid \log \lambda \leq N \}\), \(\Lambda_{N,N'} = \Lambda_{N'} \smallsetminus \Lambda_{N}\) (\(N \leq N'\))
	\item \(e^{\Lambda}, \phi^{\Lambda}\) exponential function and Drinfeld \(A\)-module associated with \(\Lambda\)
	\item \(d_{\mathbf{u}}^{Y}(\boldsymbol{\omega})\) division point/function of \(\phi^{\Lambda}\), where \(\Lambda = Y_{\boldsymbol{\omega}}\) with 
	\(\boldsymbol{\omega} \in \Psi\)
	\item \(\Delta_{a}^{Y}\), \(\Delta_{(a)}^{Y}\), \(\Delta_{\mathfrak{n}}^{Y}\), \(\Delta^{Y'|Y}\) various discriminants/discriminant functions
\end{itemize}

\section{Distributions and Derived Distributions on \(\mathfrak{Y}\)}\label{Section.Distributions-and-Derived-Distributions}

For some background and motivation on distributions in the number field case, see \cite{Kubert81} Ch. I. Throughout, \(A\) is a fixed Drinfeld coefficient ring as described in
\ref{Subsection.First-introduction-of-notation}, and \(r\) a fixed natural number.

\begin{Definition}
	We let \(\mathfrak{Y} = \mathfrak{Y}^{r}\) be the set of pairs \((\mathbf{u}, Y)\) as in \eqref{Eq.Definition-of-frak-G}, where \(\mathbf{u} \in V = K^{r}\) and
	\(Y \subset V\) is an \(A\)-lattice. We call \(\mathfrak{Y}\) the \textbf{distribution domain} for the fixed data \(A\) and \(r\). Given an additively written
	abelian group \(M\), an \(M\)-valued distribution on \(\mathfrak{Y}\) is a function \(f \colon \mathfrak{Y} \to M\) subject to
	\begin{equation}\label{Eq.Function-f-depends-only-on-n-modulo-Y}
		f(\mathbf{u}, Y) \text{ depends only on the class of \(\mathbf{u}\) modulo \(Y\)};
	\end{equation}	
	if \(Y'|Y\) is a lattice pair in \(V\) and \(\mathbf{v} \in V\), then
	\begin{equation}\label{Eq.Evaluation-of-f-for-sublattice-Y'-and-Y}
		\sum_{\substack{\mathbf{u} \in V/Y \\ \mathbf{u} \equiv \mathbf{v}~(\mathrm{mod } Y')}} f(\mathbf{u}, Y) = f(\mathbf{v}, Y').
	\end{equation}
\end{Definition}

\subsection{} The prototype of a distribution on \(\mathfrak{Y}\) comes out as follows. Assume that \(m \colon V \to M\) is a function such that for each 
\((\mathbf{u}, Y) \in \mathfrak{Y}\), the infinite sum
\begin{equation}\label{Eq.Convergence-of-series-expansion-of-f}
	f(\mathbf{u}, Y) \defeq \sum_{\substack{\mathbf{x} \in V \\ \mathbf{x} \equiv \mathbf{u} ~(\mathrm{mod } Y)}} m(\mathbf{x})
\end{equation}
converges (which requires a suitable topological structure on \(M\)). Then \eqref{Eq.Function-f-depends-only-on-n-modulo-Y} and 
\eqref{Eq.Evaluation-of-f-for-sublattice-Y'-and-Y} are trivially fulfilled, and \(f\) defines a distribution. This remains true if the sum in 
\eqref{Eq.Convergence-of-series-expansion-of-f} is replaced by \(\sideset{}{^{\prime}} \sum\).

\subsection{} More specifically, for each \(\boldsymbol{\omega} \in \Psi\), \((\mathbf{u}, Y)\in \mathfrak{Y}\), and \(k \in \mathds{N}\), define
\begin{equation}
	E_{k, \mathbf{u}}^{Y}(\boldsymbol{\omega}) \defeq \sideset{}{^{\prime}} \sum_{\substack{\mathbf{x} \in V \\ \mathbf{x} \equiv \mathbf{u} ~(\mathrm{mod } Y)}} (\mathbf{x}\boldsymbol{\omega})^{{-}k},
\end{equation}
where \(\mathbf{x}\boldsymbol{\omega} = \sum_{1 \leq i \leq r} x_{i}\omega_{i} \in C_{\infty}\). The sum converges and defines a holomorphic function \(E_{k, \mathbf{u}}^{Y}\)
on \(\Psi\), the \textbf{Eisenstein series} of weight \(k\) and shape \((\mathbf{u}, Y)\). In fact, \(E_{k, \mathbf{u}}^{Y}\) is a modular form of weight \(k\) for
a suitable congruence subgroup of \(\Gamma_{Y} = \GL(Y)\). For details, see e.g. \cite{Goss80}, \cite{BassonBreuerPink-tA}, or \cite{gekeler2023drinfeld}. Therefore, 
the system \(\{ E_{k, \mathbf{u}}^{Y}\}\) with \(k\) fixed describes a distribution with values in the \(C_{\infty}\)-algebra of modular forms.

\subsection{} Let \(f \colon \mathfrak{Y} \to M\) be a distribution. Its \textbf{derivative} \(Df = g\) is the function \(g \colon \mathfrak{Y} \to M\) defined by
\begin{equation}
	g(\mathbf{u}, Y) \defeq f(\mathbf{u}, Y) - f(0, Y).
\end{equation}
Since \(g \equiv 0\) on \(\mathfrak{Y} \smallsetminus \mathfrak{Y}^{*}\), we regard it as a function on 
\begin{equation}
	\mathfrak{Y}^{*} \defeq \{ (\mathbf{u}, Y) \in \mathfrak{Y} \mid \mathbf{u} \notin Y \}.
\end{equation}
If \(Y'|Y\) is a lattice pair in \(V\), we define the \textbf{discriminant} \(\Disc(g)^{Y'|Y}\) \textbf{of} \(Y'|Y\) \textbf{with respect to} \(g\) as
\begin{equation}\label{Eq.Definition-of-discriminant-of-lattice-pair-wrt-function}
	\Disc(g)^{Y'|Y} \defeq {-} \sideset{}{^{\prime}}\sum_{\mathbf{u} \in Y'/Y} g(\mathbf{u}, Y). \qquad \text{(Note the minus sign!)}
\end{equation}
By direct calculation, we find for \(\mathbf{v} \in V\):
\begin{equation}\label{Eq.Expansion-discriminant-of-lattice-pair}
	\sum_{\substack{\mathbf{u} \in V/Y \\ \mathbf{u} \equiv \mathbf{v} ~(\mathrm{mod } Y')}} g(\mathbf{u}, Y) + \Disc(g)^{Y'|Y} = g(\mathbf{v}, Y'),
\end{equation}
and for a tower \(Y''|Y'|Y\):
\begin{equation}\label{Eq.Discriminant-for-lattice-tower}
	\Disc(g)^{Y''|Y} = \Disc(g)^{Y''|Y} + [Y'':Y'] \Disc(g)^{Y'|Y}.
\end{equation}
Moreover,
\begin{equation}
	\Disc(g)^{Y'|Y} = [Y':Y] f(0,Y) - f(0, Y').
\end{equation}
Note that the definition \eqref{Eq.Definition-of-discriminant-of-lattice-pair-wrt-function} of \(\Disc(g)\) as well as \eqref{Eq.Expansion-discriminant-of-lattice-pair}
and \eqref{Eq.Discriminant-for-lattice-tower} involve only \(g\) but not \(f\), and \eqref{Eq.Discriminant-for-lattice-tower} is a consequence of 
\eqref{Eq.Expansion-discriminant-of-lattice-pair} and \eqref{Eq.Function-f-depends-only-on-n-modulo-Y} for \(g\) without reference to \(f\).

\begin{Definition}
	A \textbf{derived distribution} on \(\mathfrak{Y}\) with values in \(M\) is a function \(g \colon \mathfrak{Y}^{*} \to M\) subject to 
	\eqref{Eq.Function-f-depends-only-on-n-modulo-Y}	, i.e., \(g(\mathbf{u},Y)\) depends only on the class of \(\mathbf{u}\) modulo \(Y\), and
	\eqref{Eq.Expansion-discriminant-of-lattice-pair}, where \(\Disc(g)^{Y'|Y}\) is defined in \eqref{Eq.Definition-of-discriminant-of-lattice-pair-wrt-function}
	(and thus to \eqref{Eq.Discriminant-for-lattice-tower}, by the remark above).
\end{Definition}

It is a major problem to decide whether a given derived distribution \(g\) comes in fact as the derivative of a distribution \(f\) (a \textbf{primitive} of \(g\)) as above.
We note the trivial observation:

\begin{Lemma}
	\begin{enumerate}[label=\(\mathrm{(\roman*)}\)]
		\item Both the sets of distributions \(f\) and of derived distributions \(g\) are abelian groups, and \(f \mapsto Df = g\) is a homomorphism.
		\item If \(M\) has no \(p\)-torsion then a primitive \(f\) of \(g\) is unique up to adding a constant distribution. Here a distribution \(f\) is \textbf{constant}
		if \(f(\mathbf{u}, Y) = f(0, Y)\) and \(f(0, Y') = [Y':Y] f(0,Y)\) for all \(\mathbf{u} \in V\) and lattice pairs \(Y'|Y\).
	\end{enumerate}	
\end{Lemma}

For the next definition, we suppose that \(M\) is a Hausdorff topological group, so that the occurring limits are meaningful.

\begin{Definition}
	\begin{enumerate}[label=\(\mathrm{(\roman*)}\)]
		\item Let \(f \colon \mathfrak{Y} \to M\) be a distribution. It is \textbf{motivated by} \(m \colon V \smallsetminus \{0\} \to M\) if there exists a norm
		\(\lVert \ldot \rVert\) on \(V_{\infty}\) such that
		\begin{equation} \label{Eq.Series-expansion-for-distribution}
			f(\mathbf{u}, Y) = \lim_{N \to \infty} \sideset{}{^{\prime}} \sum_{\substack{\mathbf{x} \in V \\ \mathbf{x} \equiv \mathbf{u}~(\mathrm{mod } Y) \\ \log_{q} \lVert \mathbf{x} \rVert \leq N}} m(\mathbf{x}).
		\end{equation}
		\item A derived distribution \(g \colon \mathfrak{Y}^{*} \to M\) is \textbf{motivated by} \(m\) if 
		\begin{equation} \label{Eq.Derived-distribution-motivated-by-m}
			g(\mathbf{u}, Y) = \lim_{N \to \infty} \Big( \sideset{}{^{\prime}} \sum_{\substack{\mathbf{x} \in V \\ \mathbf{x} \equiv \mathbf{u} ~(\mathrm{mod } Y) \\ \log_{q} \lVert \mathbf{x} \rVert \leq N}} m(\mathbf{x}) - \sideset{}{^{\prime}} \sum_{\substack{\mathbf{y} \in Y \\ \log_{q} \lVert \mathbf{y} \rVert \leq N}} m(\mathbf{y}) \Big).
		\end{equation}
	\end{enumerate}
\end{Definition}

\begin{Remarks-nn}
	\begin{enumerate}[wide, label=(\roman*)]
		\item As one easily verifies, the choice of the norm \(\lVert \ldot \rVert\) on the finite-dimensional \(K_{\infty}\)-vector space \(V_{\infty}\) is irrelevant,
		as all norms are equivalent. Typical norms on \(V_{\infty}\) are the \(\lVert \ldot \rVert_{\boldsymbol{\omega}}\) with \(\boldsymbol{\omega} \in \Psi\), where
		\(\lVert \mathbf{x} \rVert_{\boldsymbol{\omega}} = \lvert \mathbf{x} \boldsymbol{\omega} \rvert\).
		\item If the distribution \(f\) is motivated by \(m\), so is its derivative \(g = Df\).
		\item Given \(m \colon V \smallsetminus \{0\} \to M\) such that the limits in \eqref{Eq.Series-expansion-for-distribution} (resp. 
		\eqref{Eq.Derived-distribution-motivated-by-m}) always exist, it defines a distribution (resp. derived distribution) on \(\mathfrak{Y}\).
		\item If the topology on \(M\) is discrete, each limit \(\lim_{N \to \infty} x_{N}\) is \textbf{stationary}, that is, \(\lim_{N \to \infty} x_{N} = x_{N}\) for
		\(N\) sufficiently large.
	\end{enumerate}	
\end{Remarks-nn}

We will see at once that the division forms \(d_{\mathbf{u}}^{Y}(\boldsymbol{\omega})\) define a derived distribution. Later it will come out that it is motivated
by \(m \colon \mathbf{x} \mapsto \mathbf{x}\boldsymbol{\omega}\), but in general, it isn't the derivative of a distribution.

\subsection{}\label{Subsection.Collection-of-facts}%
At several occasions, we will make use of the following easily proved facts. Let \(a,b\) be natural numbers with \(c \defeq \gcd(a,b)\). Then
\begin{equation} \label{Eq.gcd-condition}
	\gcd(q^{a} - 1, q^{b} - 1) = q^{c} - 1.
\end{equation}
Let \(a_{i} \in C_{\infty}^{*}\) be finitely many elements with assigned weights \(k_{i} \in \mathds{Z}\) and the property that
\[
	\prod a_{i}^{n_{i}} = 1 \qquad \text{whenever} \qquad \sum n_{i}k_{i} = 0 \quad (n_{i} \in \mathds{Z}).
\]
Then there exists a unique \(b \in C_{\infty}^{*}\) such that
\begin{equation} \label{Eq.Existence-of-b-infinity}
	a_{i} = b^{k_{i}/k} \quad (k \defeq \gcd(k_{i})),
\end{equation}
and \(b\) lies in the group generated by \(\{a_{i}\}\).

\section{Basic properties of the functions \(d_{\mathbf{u}}^{Y}\) and \(\Delta^{Y'|Y}\)} \label{Section.Basic-properties-of-the-functions}

\begin{Proposition}
	Let a lattice pair \(Y'|Y\) and \(\mathbf{v} \in V \smallsetminus Y'\) be given. Then the identity
	\begin{equation}
		\Delta^{Y'|Y} \prod_{\substack{\mathbf{u} \in V/Y \\ \mathbf{u} \equiv \mathbf{v}~(\mathrm{mod } Y')}} d_{\mathbf{u}}^{Y} = d_{\mathbf{v}}^{Y'}
	\end{equation}
	holds.
\end{Proposition}

\begin{proof}
	Fix \(\boldsymbol{\omega} \in \Psi\) and let \(\varphi\) be the unique normalized isogeny from the Drinfeld module \(\phi = \phi^{Y_{\boldsymbol{\omega}}}\) to
	\(\phi' = \phi^{Y_{\boldsymbol{\omega}}'}\). Then
	\begin{equation}
		e^{Y_{\boldsymbol{\omega}}'} = \varphi \circ e^{Y_{\boldsymbol{\omega}}}
	\end{equation}
	and, according to \ref{Subsection.Assumptions-on-lattice-and-fixed-element} and \ref{Subsection.Lattice-rank},
	\begin{equation}
		\varphi(X) = X \sideset{}{^{\prime}} \prod_{\mathbf{w} \in Y'/Y} \big(1 - (d_{\mathbf{w}}^{Y})^{-1} X \big) = \Delta^{Y'|Y}(\boldsymbol{\omega}) \prod_{\mathbf{w} \in Y'/Y} (X - d_{\mathbf{w}}^{Y})
	\end{equation}
	(where \(d_{0}^{Y} = 0\)). The \(\mathbf{u} \in V/Y\) with \(\mathbf{u} \equiv \mathbf{v} \pmod{Y'}\) are obtained by \(\mathbf{u} = \mathbf{v} - \mathbf{w}\) with
	\(\mathbf{w} \in Y'/Y\). Hence, upon inserting \(d_{\mathbf{v}}^{Y}(\boldsymbol{\omega}) = e^{Y_{\boldsymbol{\omega}}}(\mathbf{v}\boldsymbol{\omega})\), the right hand
	side gives
	\begin{align*}
		\Delta^{Y'|Y}(\boldsymbol{\omega}) \prod_{\mathbf{w} \in Y'/Y} \big( e^{Y_{\boldsymbol{\omega}}}(\mathbf{v}\boldsymbol{\omega}) - e^{Y_{\boldsymbol{\omega}}}(\mathbf{w} \boldsymbol{\omega}) \big) 	&= \Delta^{Y'|Y}(\boldsymbol{\omega}) \prod_{\substack{\mathbf{u} \in V \\ \mathbf{u} \equiv \mathbf{v}~(\mathrm{mod } Y')}} e^{Y_{\boldsymbol{\omega}}}(\mathbf{u}\boldsymbol{\omega}) \\
										&= \Delta^{Y'|Y}(\boldsymbol{\omega}) \prod_{\mathbf{u}} d_{\mathbf{u}}^{Y_{\boldsymbol{\omega}}},
	\end{align*}
	while
	\[
		\varphi(e^{Y_{\boldsymbol{\omega}}}(\mathbf{v} \boldsymbol{\omega})) = e^{Y_{\boldsymbol{\omega}}'}(\mathbf{v}\boldsymbol{\omega}) = d_{\mathbf{v}}^{Y'}(\boldsymbol{\omega}).
	\]
\end{proof}

\subsection{} The proposition states that \(g(\mathbf{u}, Y) \defeq d_{\mathbf{u}}^{Y}\) describes a derived distribution with values in the multiplicative group
\(C_{\infty}^{*}\) (if the argument \(\boldsymbol{\omega}\) of \(d_{\mathbf{u}}^{Y}\) is fixed), or in the multiplicative group of holomorphic invertible functions
\(\mathcal{O}(\Psi)^{*}\) on \(\Psi\) (if \(d_{\mathbf{u}}^{Y}\) is regarded as a function on \(\Psi\)).

We keep this dichotomy in what follows (regarding \(\boldsymbol{\omega}\) either as fixed or as a variable on \(\Psi\)), but write in general for simplicity only the case
of a constant \(\boldsymbol{\omega}\). Then \(\Delta^{Y'|Y} = (\sideset{}{^{\prime}} \prod _{\mathbf{u} \in Y'/Y} d_{\mathbf{u}}^{Y})^{-1}\) becomes the discriminant
\(\Disc(g)^{Y'|Y}\) as formally defined in \eqref{Eq.Definition-of-discriminant-of-lattice-pair-wrt-function}, and from \eqref{Eq.Expansion-discriminant-of-lattice-pair} we find
\begin{equation}\label{Eq.Delta-form-on-lattice-tower}
	\Delta^{Y''|Y} = \Delta^{Y''|Y'} \cdot (\Delta^{Y'|Y})^{[Y'':Y]}
\end{equation}
if \(Y''|Y'|Y\).

Next we give an application of the derived distribution property of the \(d_{*}^{*}\) to the multiplicative Hecke action on the forms \(\Delta^{Y'|Y}\).

\subsection{}\label{Subsection.Lattices-and-grassmanians}%
Let \(\mathfrak{p}\) be a prime ideal of \(A\) coprime with the lattice pair \(Y'|Y\), which means that \(\mathfrak{p}\) doesn't divide the Euler-Poincaré characteristic
\(\chi(Y'|Y)\) of the finite \(A\)-module \(Y'|Y\). We write \((\mathfrak{p}, \chi(Y'|Y)) = 1\) for this property. Then canonically,
\begin{equation}
	Y/\mathfrak{p}Y \overset{\cong}{\longrightarrow} Y'/\mathfrak{p}Y',
\end{equation}
which is an \(r\)-dimensional vector space over the finite field \(\mathds{F}_{\mathfrak{p}} \defeq A/\mathfrak{p}\). An \(A\)-lattice \(Z\) with 
\(\mathfrak{p}Y \subset Z \subset Y\) has \textbf{type} \(i\) \textbf{for} \((Y, \mathfrak{p})\) if
\begin{equation}
	\dim_{\mathds{F}_{\mathfrak{p}}}(Z/\mathfrak{p}Y) = i \quad (0 \leq i < r).
\end{equation}
Then
\begin{equation}\label{Eq.Bijection-of-lattices-of-type-i}
	Z \longmapsto Z' \defeq Z + \mathfrak{p}Y'
\end{equation}
is a well-defined bijection from the set \(\mathcal{L}(Y, \mathfrak{p}, i)\) of lattices of type \(i\) for \((Y, \mathfrak{p})\) to \(\mathcal{L}(Y', \mathfrak{p}, i)\), with
inverse \(Z' \mapsto Z \defeq Z' \cap Y\). We have
\begin{equation}\label{Eq.Relation-of-cardinalities-for-lattices-and-grassmanians}
	\# \mathcal{L}(Y, \mathfrak{p}, i) = \# \Gr_{\mathds{F}_{\mathfrak{p}}}(r,i) \eqdef c_{i}^{(r)}(\mathfrak{p}),
\end{equation}
where \(\Gr_{\mathds{F}_{\mathfrak{p}}}(r,i)\) is the Grassmannian of the \(i\)-subspaces of \(\mathds{F}_{\mathfrak{p}}^{r}\). Its cardinality \(c_{i}^{(r)}(\mathfrak{p})\)
is given by a well-known formula (\cite{Shimura71} Ch. III Proposition 3.18; replace \(p\) with \(q_{\mathfrak{p}} \defeq \# \mathds{F}_{\mathfrak{p}}\)), for which we currently 
have no use.

The \(\mathfrak{p}\)-\textbf{th Hecke correspondence of type} \(i\) (\(0 < i < r\)) is the set-valued map \(Y \mapsto \mathcal{L}(Y, \mathfrak{p}, i)\). Regarding a 
modular form \(f\) for \(\Gamma_{Y} = \GL(Y)\) as a certain function \(f(\boldsymbol{\omega}) = f(Y_{\boldsymbol{\omega}})\) on the set of lattices 
\(\Lambda = Y_{\boldsymbol{\omega}}\) isomorphic with \(Y\), the \(\mathfrak{p}\)-\textbf{th Hecke Operator of type} \(i\) acts on \(f\) through
\begin{equation}\label{Eq.Definition-of-Hecke-operator-of-type-i}
	T_{\mathfrak{p},i}(f)(Y_{\boldsymbol{\omega}}) = \sum_{Z \in \mathcal{L}(Y,\mathfrak{p},i)} f(Z_{\boldsymbol{\omega}}).
\end{equation}
(In the notation of Shimura \cite{Shimura71} III, \(T_{\mathfrak{p},i}\) would be written as \(T(1,\dots,1,\mathfrak{p}, \dots, \mathfrak{p})\) with \(i\) 1's and 
\((r-i)\) \(\mathfrak{p}\)'s.)

\subsection{} We postpone the study of Hecke operators in the above sense to possible future work. Instead, we investigate products analogous with 
\eqref{Eq.Definition-of-Hecke-operator-of-type-i}, and where \(f\) is a discriminant \(\Delta^{Y'|Y}\). That is, we are concerned with a multiplicative Hecke operator
\(T_{\mathfrak{p},i}^{*}\), where
\begin{equation}
	T_{\mathfrak{p},i}^{*}(f)(Y_{\boldsymbol{\omega}}) = \prod_{Z \in \mathcal{L}(Y,\mathfrak{p},i)} f(Z_{\boldsymbol{\omega}}).
\end{equation}
Assuming the framework of \ref{Subsection.Lattices-and-grassmanians}, where \(\mathfrak{p}\) is a prime of degree \(d \in \mathds{N}\), there is the following result.

\begin{Theorem} \label{Theorem.Action-of-Hecke-operator-on-Delta-form-for-lattice-pair}
	The action of \(T_{\mathfrak{p},i}^{*}\) on \(\Delta^{Y'|Y}\) is given by
	\begin{equation}\label{Eq.Theorem-Hecke-operator-in-terms-of-Delta-forms}
		T_{\mathfrak{p},i}^{*}(\Delta^{Y'|Y}) = \left( \frac{\Delta^{\mathfrak{p} Y'|\mathfrak{p}Y}}{\Delta^{Y'|Y}} \right)^{e} (\Delta^{Y'|Y})^{c_{i}^{(r)}(\mathfrak{p})}
	\end{equation}	
	with the exponent \(c_{i}^{(r)}(\mathfrak{p})\) of \eqref{Eq.Relation-of-cardinalities-for-lattices-and-grassmanians} and
	\begin{equation}
		e = c_{i}^{(r)}(\mathfrak{p}) q^{d_{i}}(q^{d(r-i)} - 1)(q^{dr} - 1)^{{-}1}.
	\end{equation}
\end{Theorem}

\begin{proof}
	\begin{enumerate}[wide, label=(\roman*)]
		\item We will use the so-far established properties of the \(d_{*}^{*}\) and \(\Delta^{**}\) and the geometry of \(Y/\mathfrak{p}Y \cong \mathds{F}_{\mathfrak{p}}^{r}\).
		The frequently occurring \(c_{i}^{(r)}(\mathfrak{p})\) will be abbreviated by \(c\). For \(Z \in \mathcal{L}(Y, \mathfrak{p}, i)\), we let 
		\(Z' \in \mathcal{L}(Y', \mathfrak{p}, i)\) be as in \eqref{Eq.Bijection-of-lattices-of-type-i}. The diagram of inclusions of lattices
		\begin{equation}
			\begin{tikzcd}
				\mathfrak{p}Y' \ar[r, hook]					& Z' \ar[r, hook]				& Y' \\
				\mathfrak{p}Y \ar[u, hook] \ar[r, hook]		& Z \ar[u, hook] \ar[r, hook]	& Y \ar[u, hook]
			\end{tikzcd}
		\end{equation}
		will be crucial.
		\item We calculate 
		\begin{equation}
			\frac{T_{\mathfrak{p},i}^{*}(\Delta^{Y'|Y})}{(\Delta^{Y'|Y})^{c}} = \prod_{Z \in \mathcal{L}(Y, \mathfrak{p}, i)} \left( \frac{\Delta^{Z'|Z}}{\Delta^{Y'|Y}} \right),
		\end{equation}
		taking into account that \(\Delta^{Z'|Z}\) and \(\Delta^{Y'|Y}\) may be written as products of division forms \(d_{\mathbf{u}}^{\mathfrak{p}Y}\).
		\item First, we choose representatives for \(Y'/\mathfrak{p}Y\) as follows. From \((\mathfrak{p}, \chi(L'/L)) = 1\), the \(A\)-module \(Y'/\mathfrak{p}Y\) splits
		canonically into its \(\mathfrak{p}\)-primary part \(Y/\mathfrak{p}Y\) and its non-\(\mathfrak{p}\)-part \(\overline{Q}\), which maps isomorphically onto
		\(Y'/Y\). Let \(P\) be an \(\mathds{F}\)-subvector space of \(Y\) complementary with \(\mathfrak{p}Y\), which we endow via \(P \overset{\simeq}{\to} Y/\mathfrak{p}Y\)
		with an \(\mathds{F}_{\mathfrak{p}}\)-structure. In particular, \(Z \in \mathcal{L}(Y, \mathfrak{p}, i)\) corresponds to an \(\mathds{F}_{\mathfrak{p}}\)-subspace,
		labelled \(\overline{Z}\), of dimension \(i\) of \(P \cong \mathds{F}_{\mathfrak{p}}^{r}\). Further, we let \(Q\) be an \(\mathds{F}\)-vector space system of 
		representatives for \(\overline{Q}\) in \(Y'\). Then
		\begin{equation}
			Y' = \mathfrak{p}Y \oplus P \oplus Q,
		\end{equation}
		and for each \(Z\), its associated \(Z'\) is \(Z \oplus Q\).
		\item Now 
		\[
			(\Delta^{Y'|Y})^{-1} = \sideset{}{^{\prime}} \prod_{\mathbf{u} \in Q} d_{\mathbf{u}}^{Y} = \Big( \sideset{}{^{\prime}} \prod_{\mathbf{u} \in Q} \prod_{\mathbf{v} \in P} d_{\mathbf{u} + \mathbf{v}}^{\mathfrak{p}Y} \Big)(\Delta^{Y|\mathfrak{p}Y})^{\#Q-1}.
		\]
		For each subset \(W\) of \(P \oplus Q \smallsetminus \{0\}\), put
		\begin{equation}
			\pi(W) \defeq \prod_{\mathbf{w} \in W} d_{\mathbf{w}}^{\mathfrak{p}Y}.
		\end{equation}
		Thus the above becomes
		\begin{equation}\label{Eq.Inverse-of-Delta-form-of-lattice-pair}
			(\Delta^{Y'|Y})^{-1} = \pi(P \oplus Q \smallsetminus P) \pi(P \smallsetminus \{0\})^{1 - \#Q}.
		\end{equation}
		\item We aim to write \(\prod_{Z} \Delta^{Z'|Z}\) in the same format. For \(\mathbf{u} \in Q \smallsetminus \{0\}\),
		\[
			d_{\mathbf{u}}^{Z} = \Big( \prod_{\substack{\mathbf{v} \in P \oplus Q \\ \mathbf{v} \equiv \mathbf{u}~(\mathrm{mod } \mathfrak{p}Y)}} d_{\mathbf{v}}^{\mathfrak{p}Y} \Big) \Delta^{Z|\mathfrak{p}Y} = \prod_{\mathbf{v} \in \overline{Z}} d_{\mathbf{u} + \mathbf{v}}^{\mathfrak{p}Y}(\pi(\overline{Z} \smallsetminus \{0\}))^{-1},
		\]
		and so
		\begin{align}\label{Eq.Inverseo-of-product-of-Delta-forms}
			\Big( \prod_{Z \in \mathcal{L}(Y, \mathfrak{p}, i)} \Delta^{Z'|Z} \Big)^{-1} 	&= \prod_{Z} \sideset{}{^{\prime}} \prod_{\mathbf{u} \in Z'/Z} d_{\mathbf{u}}^{Z} \\
																							&= \prod_{Z} \sideset{}{^{\prime}} \prod_{\mathbf{u} \in Q} \prod_{\mathbf{v} \in \overline{Z}} d_{\mathbf{u} + \mathbf{v}}^{\mathfrak{p}Y} / \Big( \prod_{Z} \pi(\overline{Z} \smallsetminus \{0\})^{\#Q - 1} \Big). \nonumber
		\end{align}
		\item Let us first treat the denominator of \eqref{Eq.Inverseo-of-product-of-Delta-forms}. In the product 
		\[
			\prod_{Z} \pi(\overline{Z} \smallsetminus \{0\}) = \prod_{Z} \sideset{}{^{\prime}} \prod_{\mathbf{v} \in \overline{Z}} d_{\mathbf{v}}^{\mathfrak{p}Y} = \sideset{}{^{\prime}} \prod_{\mathbf{w} \in P} (d_{\mathbf{w}}^{\mathfrak{p}Y})^{m},
		\]
		each subscript \(0 \neq \mathbf{w} \in P\) appears with multiplicity
		\begin{equation}
			m = c \cdot (q^{d_{i}} - 1)/(\# P -1)
		\end{equation}
		independently of \(\mathbf{w}\). Here \(c = c_{i}^{(r)}(\mathfrak{p}) = \#\{Z\}\), \(q^{d_{i}} = \#\overline{Z}\), \(\# P = q^{d_{r}}\). Hence the denominator
		of \eqref{Eq.Inverseo-of-product-of-Delta-forms} is
		\begin{equation}
			\prod_{Z} \pi(\overline{Z} \smallsetminus \{0\})^{\#Q - 1} = \pi(P \smallsetminus \{0\})^{m \cdot(\#Q - 1)}.
		\end{equation}
		\item The numerator of \eqref{Eq.Inverseo-of-product-of-Delta-forms} is
		\[
			\prod_{Z} \sideset{}{^{\prime}} \prod_{\mathbf{u} \in Q} \prod_{\mathbf{v} \in \overline{Z}} d_{\mathbf{u} + \mathbf{v}}^{\mathfrak{p}Y} = \sideset{}{^{\prime}} \prod_{\mathbf{w} \in P \oplus Q} (d_{\mathbf{w}}^{\mathfrak{p}Y})^{m(\mathbf{w})}
		\]
		with certain multiplicities \(m(\mathbf{w})\). Since \(\mathbf{u} \neq 0\), the subscript \(\mathbf{w} = \mathbf{u} + \mathbf{v}\) always lies in 
		\(P \oplus Q \smallsetminus P\). The value \(\mathbf{w} = \mathbf{u}\) with \(\mathbf{v} = 0\) appears precisely once for each \(Z\). Hence
		\(m(\mathbf{w}) = c = \#\{Z\}\) if \(\mathbf{w} = \mathbf{u} \in Q \smallsetminus \{0\}\). On the other hand, if \(\mathbf{w} = \mathbf{u} + \mathbf{v}\) with
		\(\mathbf{v} \neq 0\) (i.e., \(\mathbf{w} \in P \oplus Q \smallsetminus P \smallsetminus Q\)), then \(m(\mathbf{w}) = c\cdot (q^{di} - 1)/(q^{dr} - 1)\),
		since the \(\#\{Z\}(\#(Q)-1)(q^{di}-1)\) values of \(\mathbf{w} = \mathbf{u} + \mathbf{v}\) with \(\mathbf{v} \neq 0\) in the triple product are evenly distributed
		over the \(\#(P \oplus Q \smallsetminus P \smallsetminus Q) = (\#Q - 1)(\#P - 1)\) elements of \(P \oplus Q \smallsetminus P \smallsetminus Q\). So the numerator is
		\begin{equation}\label{Eq.Numerator-of-the-term-for-the-inverse-of-the-product-of-Delta-forms}
			\prod_{Z} \sideset{}{^{\prime}}\prod_{\mathbf{u} \in Q} \prod_{\mathbf{v} \in \overline{Z}} d_{\mathbf{u} + \mathbf{v}}^{\mathfrak{p}Y} = \pi(Q \smallsetminus \{0\})^{c} \pi(P \oplus Q \smallsetminus P \smallsetminus Q)^{c(q^{di}-1)/(q^{dr}-1)}
		\end{equation}
		with \(c = c_{i}^{(r)}(\mathfrak{p})\).
		\item Combining formulas \eqref{Eq.Inverse-of-Delta-form-of-lattice-pair} to \eqref{Eq.Numerator-of-the-term-for-the-inverse-of-the-product-of-Delta-forms}, the 
		obvious 
		\begin{equation}
			\pi(W_{1} \cupdot W_{2}) = \pi(W_{1}) \pi(W_{2})
		\end{equation}
		for \(W_{1} = P \oplus Q \smallsetminus P \smallsetminus Q\), \(W_{2} = P \smallsetminus \{0\}\), and simplifying, one finds
		\begin{equation}\label{Eq.Product-formula-over-lattices}
			\prod_{Z \in \mathcal{L}(Y, \mathfrak{p}, i)} \frac{\Delta^{Z'|Z}}{\Delta^{Y'|Y}} = \pi(P \oplus Q \smallsetminus P \smallsetminus Q)^{c\ell} \pi(P \smallsetminus \{0\})^{{-}c\ell(\#Q - 1)}
		\end{equation}
		with \(\ell \defeq 1 - (q^{d_{i}}-1)/(q^{d_{r}}-1) = q^{di}(q^{d(r-i)} - 1)/(q^{dr}-1)\).
		\item It remains to evaluate
		\begin{align*}
			\pi(P \oplus Q \smallsetminus P \smallsetminus Q)/\pi(P \smallsetminus \{0\})^{\#Q - 1}	&= \frac{\pi(P \oplus Q \smallsetminus \{0\}) \pi^{-1}(Q \smallsetminus \{0\})}{\pi(P \smallsetminus \{0\})^{\#Q}} \\
								&= (\Delta^{Y|\mathfrak{p}Y})^{\# Q} (\Delta^{Y'|\mathfrak{p}Y'})^{-1} \Delta^{\mathfrak{p}Y'|\mathfrak{p}Y} \\
								&= \Delta^{\mathfrak{p}Y'|\mathfrak{p}Y}/\Delta^{Y'|Y}	
		\end{align*}
		by \eqref{Eq.Delta-form-on-lattice-tower} and \([Y':Y] = \#Q\). Together with \eqref{Eq.Product-formula-over-lattices}, this gives
		\[
			\prod_{Z} \frac{\Delta^{Z'|Z}}{\Delta^{Y'|Y}} = \left( \frac{\Delta^{\mathfrak{p}Y'|\mathfrak{p}Y}}{\Delta^{Y'|Y}} \right)^{c\ell}
		\]
		with \(c = c_{i}^{(r)}(\mathfrak{p})\) and \(c\ell = e\) as stated in \eqref{Eq.Theorem-Hecke-operator-in-terms-of-Delta-forms}.
	\end{enumerate}
\end{proof}

\begin{Corollary}\label{Corollary.Action-of-Hecke-operators-for-principal-ideals}
	In the situation of the theorem, assume that \(\mathfrak{p}\) is principal, generated by \(n \in A\).\footnote{The letter \(p\) is occupied by 
	\(p = \Characteristic(\mathds{F})\).} Then 
	\begin{equation}
		T_{\mathfrak{p},i}^{*}(\Delta^{Y'|Y}) = n^{e'}(\Delta^{Y'|Y})^{c}
	\end{equation}	
	with the exponent \(e' = (1 - [Y':Y])e\) and \(e\) and \(c = c_{i}^{(r)}(\mathfrak{p})\) as in \eqref{Eq.Theorem-Hecke-operator-in-terms-of-Delta-forms}.
\end{Corollary}

\begin{proof}
	We use \(\mathfrak{p}Y' = nY'\), \(\mathfrak{p}Y = nY\) and the fact that \(\Delta^{Y'|Y}\) as a modular form has weight \([Y':Y]-1\). This gives
	\(\Delta^{nY'|nY} = n^{1 - [Y':Y]}\Delta^{Y'|Y}\), and thus the assertion.
\end{proof}

\begin{Remarks} \label{Remarks.Essentiality-of-primality}
	\begin{enumerate}[wide, label=(\roman*)]
		\item The assumption in \ref{Subsection.Lattices-and-grassmanians} that \(\mathfrak{p}\) is a prime ideal is not essential, and was only made for reasons of
		presentation. Results generalizing Theorem \ref{Theorem.Action-of-Hecke-operator-on-Delta-form-for-lattice-pair} can be obtained for multiplicative Hecke
		operators \(T_{\mathfrak{n},i}^{*}\) for arbitrary ideals \(\mathfrak{n} \subset A\).
		\item Consider \(\Delta^{Y'|Y}\) as a modular form for the group \(\Gamma_{Y'|Y} = \GL(Y) \cap \GL(Y')\), defined on \(\Omega = C_{\infty}^{*} \backslash \Psi\).
		Let
		\[
			P \colon \mathcal{O}(\Omega)^{*} \longrightarrow \mathbf{H}(\mathcal{BT}, \mathds{Z})
		\]
		be the van der Put transform from the multiplicative group of invertible holomorphic functions on \(\Omega\) to the additive group of \(\mathds{Z}\)-valued harmonic
		\(1\)-cochains on the Bruhat-Tits building \(\mathcal{BT} = \mathcal{BT}^{r}\) of \(\PGL(r, K_{\infty})\), see \cite{Gekeler20} for details. It is 
		\(\GL(r, K_{\infty})\)-equivariant and defines a short exact sequence
		\[
			\begin{tikzcd}
				1 \ar[r]		& C_{\infty}^{*} \ar[r]	& \mathcal{O}(\Omega)^{*} \ar[r, "P"] & \mathbf{H}(\mathcal{BT}, \mathds{Z}) \ar[r]	& 0.
			\end{tikzcd}
		\]
		Then \(P(\Delta^{Y'|Y})\) belongs to the invariants \(\mathbf{H}(\mathcal{BT}, \mathds{Z})^{\Gamma_{Y'|Y}}\). As 
		\(T_{\mathfrak{p},i}^{*}(\Delta^{Y'|Y}) = \mathrm{const.} (\Delta^{Y'|Y})^{c}\) if \(\mathfrak{p} = (n)\) is principal, \(P(\Delta^{Y'|Y})\) is an eigenform under
		the additive Hecke operator \(T_{\mathfrak{p}, i}\) that acts on \(\mathbf{H}(\mathcal{BT}, \mathds{Z})^{\Gamma_{Y'|Y}}\), with eigenvalue 
		\(c = c_{i}^{(r)}(\mathfrak{p})\). It is a remarkable fact that this eigenvalue depends only on \(\lvert \mathfrak{p} \rvert = q^{d}\) and \(i\), but
		neither on \(\mathfrak{p}\) itself nor on \(Y'|Y\).
	\end{enumerate}	
\end{Remarks}

\begin{Example}
	We present the most simple example to which Theorem \ref{Theorem.Action-of-Hecke-operator-on-Delta-form-for-lattice-pair} applies, namely the case where as in
	\ref{Subsubsection.Simple-example-C-projective-line-infty-place-at-infinity}
	\begin{equation}
		A = \mathds{F}[T], \quad r = 2, \quad Y = A^{2} \quad \text{and} \quad Y' = T^{-1}Y.
	\end{equation}	
	Then \(\Delta^{Y'|Y} = \Delta_{(T)} = T^{-1}\Delta\), where \(\Delta = \Delta_{T}\) is the \enquote{usual} discriminant studied, e.g., in \cite{Gekeler88}. Let \(\mathfrak{p} = (n)\)
	be a prime different from \((T)\) (which is no restriction as \(\Delta_{T} = \Delta_{T-1}\)), of degree \(d = \deg \mathfrak{p}\). Let 
	\(\{ \mathbf{e}_{1}, \mathbf{e}_{2}\}\) be the standard basis of \(Y\). The \(c_{1}^{(2)}(\mathfrak{p}) = q^{d}+1\) elements \(Z\) of 
	\(\mathcal{L}(Y,\mathfrak{p}) \defeq \mathcal{L}(Y,\mathfrak{p},1)\) are the sublattices \(Z_{a}\) (\(a \in A \mid \deg a < d\)) and \(Z_{\infty}\), where
	\begin{equation}
		Z_{a} = A(\mathbf{e}_{1} + a\mathbf{e}_{2}) \oplus An \mathbf{e}_{2}, \quad Z_{\infty} = An\mathbf{e}_{1} \oplus A\mathbf{e}_{2}.
	\end{equation}
	By Corollary \ref{Corollary.Action-of-Hecke-operators-for-principal-ideals}, the operator \(T_{\mathfrak{p}}^{*} \defeq T_{\mathfrak{p},1}^{*}\) acts as
	\begin{equation}
		T_{\mathfrak{p}}^{*}(\Delta) = n^{e} \Delta^{q^{d}+1} \quad \text{with} \quad e = {-}(q^{2}-1)q^{d}.
	\end{equation}
	Put \(\boldsymbol{\omega} = (\omega, 1)\), where \(\omega \in \Omega = \Omega^{2} = C_{\infty} \smallsetminus K_{\infty}\), and write \(Z_{a, \boldsymbol{\omega}}\)
	for \((Z_{a})_{\boldsymbol{\omega}} = A(\omega + a) \oplus An\) and \(Z_{\infty, \boldsymbol{\omega}} = (Z_{\infty})_{\boldsymbol{\omega}} = An\omega \oplus A\).
	Then
	\begin{align*}
		\Delta(Z_{a, \boldsymbol{\omega}})	&= \Delta(nZ'_{a,\boldsymbol{\omega}}), 	&&\text{where } Z_{a,\boldsymbol{\omega}}' = n^{-1}Z_{a,\boldsymbol{\omega}} = A\left(\frac{\omega+a}{n}\right) \oplus A \\
											&= n^{1-q^{2}}\Delta(Z_{a,\boldsymbol{\omega}}') \\
											&= n^{1-q^{2}} \Delta\left( \frac{\omega + a}{n} \right) 
		\shortintertext{and}
			\Delta(Z_{\infty, \boldsymbol{\omega}})		&= \Delta(n\omega),	
	\end{align*}
	translating between functions on \(\{\text{lattices in }C_{\infty}\}\) and on \(\Omega\). Hence, as functions on \(\Omega\),
	\begin{align*}
		n^{(1-q^{2})q^{d}} \Delta^{q^{d}+1}(\omega)	&= T_{\mathfrak{p}}^{*}(Y_{\boldsymbol{\omega}}) \\
													&= \Delta(Z_{\infty, \boldsymbol{\omega}}) \prod_{a} \Delta(Z_{a, \boldsymbol{\omega}}) \\
													&= \Delta(n\omega) \prod_{a} n^{(1-q^{2})} \Delta\left(\frac{\omega +a}{n}\right) = n^{(1-q^{2})q^{d}} \Delta(n\omega) \prod_{a} \Delta\left( \frac{\omega + a}{n} \right)
	\end{align*}
	and finally the functional equation
	\begin{equation}\label{Eq.Delta-form-evaluated-on-bold-omega}
		\Delta^{q^{d}+1}(\omega) = \Delta(n\omega) \prod_{\substack{a \in A \\ \deg a < n}} \Delta \left( \frac{\omega + a}{n} \right)
	\end{equation}
	for \(\Delta\) found in \cite{Gekeler97} Corollary 2.11. Note that this formula was then a consequence of deep results about conditionally convergent Eisenstein series on the 
	Bruhat-Tits tree \(\mathcal{T} = \mathcal{BT}^{2}\) and their behavior under the Hecke operators. Here, the other way round, the far more general Theorem
	\ref{Theorem.Action-of-Hecke-operator-on-Delta-form-for-lattice-pair} implies certain Hecke eigenvalues, cf. Remarks \ref{Remarks.Essentiality-of-primality}(ii). The
	functional equation \eqref{Eq.Delta-form-evaluated-on-bold-omega} is analogous with the formula
	\begin{equation}\label{Eq.Analogy-of-Delta-form-to-elliptical-discriminant}
		\Delta^{p+1}(z) = \varepsilon(p) \Delta(pz) \prod_{0 \leq i < p} \Delta\left( \frac{z+i}{p} \right)
	\end{equation}
	for the classical elliptic discriminant \(\Delta(z)\) on the complex upper half-plane, where \(p \in \mathds{N}\) is a prime number and \(\varepsilon(p) = {-}1\) if
	\(p = 2\) and \(\varepsilon(p) = 1\) otherwise. It may be derived from Jacobi's product formula \(\Delta = (2\pi \imath)^{12} q \prod_{n \geq 1}(1 - q^{n})^{24}\);
	unfortunately, the author was unable to find a reference to the certainly well-known \eqref{Eq.Analogy-of-Delta-form-to-elliptical-discriminant}.
\end{Example}

\section{\(Z\)-functions of lattices in \(C_{\infty}\)} \label{Section.Z-functions-and-lattices}

In the whole section, \(\Lambda\) is a fixed \(A\)-lattice of rank \(r \in \mathds{N}\) in \(C_{\infty}\).

\subsection{} For \(N \in \mathds{Q}\), let
\begin{equation}
	\Lambda_{N} \defeq \{ \lambda \in \Lambda \mid \log \lambda \leq N \},
\end{equation}
a finite-dimensional \(\mathds{F}\)-vector space. There are various methods to encode the shape of \(\Lambda\), among which the \(\mathds{F}\)-spectrum 
\(\Spec_{\mathds{F}}(\Lambda)\) as in \cite{Gekeler22} 1.9, the growth function \(s \mapsto \dim_{\mathds{F}}(\Lambda_{s})\), or the \(Z\)-function \(Z_{\Lambda}\) of \(\Lambda\)
defined below, which is particularly well-suited for our present purposes. All these are equivalent.

\subsection{} Let \(S\) be an indeterminate, and choose a compatible system \(\{ S^{1/n} \mid n \in \mathds{N}\}\) of \(n\)-th roots of \(S\), that is,
\((S^{1/n})^{n/m} = S^{1/m}\) if \(m \mid n\). This is possible by Zorn's lemma. Define
\begin{equation}
	\mathds{Z}\{\{ S \}\} \defeq \bigcup_{n \in \mathds{N}} \mathds{Z}(( S^{1/n} ))
\end{equation}
as the ring of formal \textbf{Puiseux series in} \(S\) over \(\mathds{Z}\). Here, as usual, \(\mathds{Z}((X))\) is the ring of formal Laurent series in \(X\).

As \(\Lambda\) is finitely generated as an \(A\)-module, there exists a common denominator \(n \in \mathds{N}\) for the set 
\(\{ \log \lambda \mid 0 \neq \lambda \in \Lambda\} \subset \mathds{Q}\). Hence \(S^{\log \lambda}\) is a well-defined power of \(S^{1/n}\). We define the
\(Z\)-\textbf{function of } \(\Lambda\) as
\begin{equation}
	Z_{\Lambda}(S) = \sideset{}{^{\prime}} \sum_{\lambda \in \Lambda} S^{\log \lambda}
\end{equation}
and, more generally, for \(u \in K\Lambda\),
\[
	Z_{u, \Lambda}(S) = \sideset{}{^{\prime}} \sum_{\substack{\lambda \in K\Lambda \\ \lambda \equiv u~(\mathrm{mod } \Lambda)}} S^{\log \lambda}.
\]
Then \(Z_{0, \Lambda} = Z_{\Lambda}\), \(Z_{\Lambda}\) and \(Z_{u,\Lambda}\) are well-defined formal Laurent series in some \(S^{1/n}\), and therefore Puiseux series in
\(S\). Furthermore, \(Z_{\Lambda}\) and \(Z_{u,\Lambda}\) agree up to a finite number of terms, since almost all \(\lambda \in \Lambda\) satisfy \(\log \lambda > \log u\).

\subsection{} Writing \(\Lambda = Y_{\boldsymbol{\omega}}\), where \(\boldsymbol{\omega} \in \Psi\) is fixed, \(u \in K\Lambda\) equals \(\mathbf{u}\boldsymbol{\omega}\)
with \(\mathbf{u} \in V\), and the map
\begin{equation} \label{Eq.Distribution-on-frak-G}
	(\mathbf{u}, Y) \longmapsto Z_{\mathbf{u}, Y}(\boldsymbol{\omega}, S) \defeq Z_{\mathbf{u}\boldsymbol{\omega}, Y_{\boldsymbol{\omega}}}(S)
\end{equation}
satisfies the rules \eqref{Eq.Function-f-depends-only-on-n-modulo-Y} and \eqref{Eq.Evaluation-of-f-for-sublattice-Y'-and-Y}. Hence it defines a distribution on \(\mathfrak{Y}\)
with values in \(\mathds{Z}((S^{1/n}))\). (If we allow \(\boldsymbol{\omega}\) to vary over \(\Psi\), then the denominator \(n\) is unbounded, and we must replace
\(\mathds{Z}((S^{1/n}))\) with \(\mathds{Z}\{\{S\}\}\). This will however play no role in the present paper.) We also note the homogeneity property
\begin{equation} \label{Eq.Homogenity-property-of-distribution-on-frak-G}
	Z_{cu, c\Lambda}(S) = S^{\log c} Z_{u, \Lambda}(S)
\end{equation}
for \(0 \neq c \in C_{\infty}\).

Let \(\Lambda\) and \(u \in K\Lambda\) be given, and assume that \(n \in \mathds{N}\) is a common denominator of \(\log u\) and \(\log \lambda\) 
(\(0 \neq \lambda \in \Lambda\)). After some preparation, we will show:

\begin{Theorem} \label{Theorem.Distribution-on-frak-G-is-rational-function}
	\begin{enumerate}[label=\(\mathrm{(\roman*)}\)]
		\item \(Z_{u, \Lambda}(S) \in \mathds{Z}((S^{1/n}))\) is in fact a rational function in \(S^{1/n}\). More precisely,
		\begin{equation}
			Z_{u, \Lambda}(S)(1 - q_{\infty}^{r} S^{d_{\infty}}) \in \mathds{Z}[S^{1/n}, S^{{-}1/n}]
		\end{equation}
		is a Laurent polynomial with non-negative coefficients.
		\item \(Z_{u, \Lambda}(1) = {-}1\) if \(u \in \Lambda\) and \(0\) otherwise.
		\item For \(N \in \mathds{Q}\) large enough, we have the identity
		\begin{equation} \label{Eq.Theorem-Difference-of-sums-of-logarithms}
			\sideset{}{^{\prime}} \sum_{\substack{\lambda \in K\Lambda \\ \lambda \equiv u~(\mathrm{mod } \Lambda) \\ \log \lambda \leq N}} \log \lambda - \sideset{}{^{\prime}}\sum_{\lambda \in \Lambda_{N}} \log \lambda = Z_{u, \Lambda}'(1) - Z_{0, \Lambda}'(1),
		\end{equation}
		where \(Z_{*,*}' = \frac{\mathop{d}}{\mathop{dS}} Z_{**}\).
		\item For \(c \in C_{\infty}^{*}\),
		\begin{equation}
			Z_{cu, c\Lambda}'(1) \text{ equals } Z_{u, \Lambda}'(1) \text{ if } u \notin \Lambda \text{ and } Z_{u, \Lambda}'(1) - \log c \text{ if } u \in \Lambda.
		\end{equation}
	\end{enumerate}	
\end{Theorem}

\begin{Remark-nn}
	In view of (i), \(Z_{u, \Lambda}\) and its derivative may be evaluated at \(S = 1\), i.e., at \(S^{1/n} = 1\).
\end{Remark-nn}

\subsection{} For a given rational number \(N\) and \(d \in \mathds{N}\), let
\[
	N < x_{1} < x_{2} < \cdots < x_{s} \leq N + d
\]
be the rational numbers that appear in the half-open interval \((N, N+d]\) as \(x_{i} = \log \lambda\) for some \(\lambda \in \Lambda\). Let \(m_{1}, \dots, m_{s}\) be 
the respective multiplicities, \(m_{i} \defeq \#\{ \lambda \in \Lambda \mid \log \lambda = x_{i} \}\). We call \(( (x_{1},m_{1}), \cdots, (x_{s}, m_{s}))\) the
\((N,d)\)-\textbf{spectrum} \(\Spec_{(N,d)}(\Lambda)\) \textbf{of} \(\Lambda\). We say that \(\Lambda\) is \textbf{ultimately quasi-periodic with quasi-period} \(d\)
(brief: uqp-\(d\)) if there is some \(N_{0} \in \mathds{Q}\) such that for all \(N \geq N_{0}\), the following relation holds:
\begin{multline}\label{Eq.Spec-of-lattice-Lambda}\stepcounter{subsubsection}%
	\text{If } \Spec_{(N,d)}(\Lambda) = ( (x_{1},m_{1}), \dots, (x_{s},m_{s}) ), \\ \text{ then } \Spec_{(N+d,d)}(\Lambda) = ( (x_{1}+d, q^{rd}m_{1}), \dots, (x_{s} + d, q^{rd}m_{s}) ).	
\end{multline}
If this is the case for some \(d\), then we call \(N_{0}\) sufficiently large for \(\Lambda\), written \(N_{0} \gg_{\Lambda} 0\), or simply \(N_{0} \gg 0\) if \(\Lambda\)
is fixed. The following are easy to verify.
\subsubsection{} If \(\Lambda\) is uqp-\(d\), then the occurring spectral values \(x_{1}, \dots, x_{s} > N_{0}\) are well-defined modulo \(d\mathds{Z}\) (however, the numbering
may depend on \(N\));

\subsubsection{} if \(\Lambda\) is upq-\(d\) for \(d = d_{1}, d_{2} \in \mathds{N}\), then it is uqp-\(g\) for \(g \defeq \gcd(d_{1},d_{2})\);

\subsubsection{}\label{Subsubsection.Characterisation-if-lattice-rank-r-1} if \(r = \rk_{A}(\Lambda) = 1\), then \(\Lambda\) is uqp-\(d\) for each \(d = \deg a\), where 
\(a \in A \smallsetminus \mathds{F}\);

\subsubsection{}\label{Subsubsection.Being.uqp-d-is-stable}%
being uqp-\(d\) is stable under taking orthogonal direct sums of lattices in \(C_{\infty}\).

(The direct sum \(\Lambda = \oplus \Lambda^{(i)}\) of lattices \(\Lambda^{(i)}\) is \textbf{orthogonal}, written \(\bigperp \Lambda^{(i)}\), if for each 
\(\lambda \in K_{\infty}\Lambda\), \(\lambda = \sum \lambda_{i}\) with \(\lambda_{i} \in K_{\infty} \Lambda^{(i)}\), the rule 
\(\lvert \lambda \rvert = \max_{i} \lvert \lambda_{i} \rvert\) holds.)

Here \ref{Subsubsection.Characterisation-if-lattice-rank-r-1} is a consequence of the Riemann-Roch theorem, which implies that \(\ell(N) \defeq \dim_{\mathds{F}}(\Lambda_{N})\)
grows as \(\ell(N+d) = \ell(N) + d\) for \(N \gg 0\). Hence

\subsubsection{} in case \(r=1\), \(\Lambda\) is uqp-\(d_{\infty}\), as \(d_{\infty} = \gcd\{ \deg a \mid 0 \neq a \in A\}\).

\subsection{} We want to show that \(\Lambda\) is always uqp-\(d_{\infty}\). By the preceding, this was obvious if we could write \(\Lambda\) as an orthogonal sum
\(\Lambda^{(1)} \perp \cdots \perp \Lambda^{(r)}\) of one-dimensional \(\Lambda^{(i)}\). This is always possible if the base ring \(A\) is a polynomial ring
\(A = \mathds{F}[T]\), due to the existence of successive minimum bases (see \cite{Gekeler19} Sect. 3), but is delicate in general. So we reduce the wanted
assertion to the polynomial case.

\subsection{} Let \(a \in A\) be non-constant, of degree \(d \in \mathds{N}\), and put \(A_{0} \defeq \mathds{F}[a] \subset A\),
\(K_{0} \defeq \mathds{F}(a) = \Quot A_{0}\), \(K_{0, \infty} \defeq \mathds{F}((a^{-1})) =\) the completion of the rational function field \(K_{0}\)
at its infinite place \(\infty_{0}\). Consider the diagram of ring/field extensions
\begin{equation}\stepcounter{subsubsection}%
	\begin{tikzcd}
		A \ar[r, hook] \ar[d, dash]	& K \ar[r, hook] \ar[d, dash]	& K_{\infty} \ar[d, dash] \\
		A_{0} \ar[r, hook]			& K_{0} \ar[r, hook]				& K_{0, \infty}.	
	\end{tikzcd}
\end{equation}
It has the following properties:
\subsubsection{}\stepcounter{equation}%
\enquote{\(\infty\)} is the unique place of \(K\) above the place \enquote{\(\infty_{0}\)} of \(K_{0}\);

\subsubsection{}\stepcounter{equation}%
\(d = \rk_{A_{0}}(A) = [K : K_{0}] = [K_{\infty} : K_{0, \infty}] = e \cdot f\), where \(e = {-}v_{\infty}(a) = d/d_{\infty}\) is the ramification index
and \(f = d_{\infty}\) the residue class degree at \(\infty_{0}\).

Let \(\lvert \ldot \rvert_{0}\) be the absolute value on \(K_{0, \infty}\) with its natural normalization \(\lvert a \rvert_{0} = q\) and \(\log_{0}\) the logarithm
function \(\log_{0} x = \log_{q} \lvert x \rvert_{0}\) on \(C_{\infty}\). Then
\begin{equation}
	\log = d \cdot \log_{0}.
\end{equation}
Now \(\rk_{A_{0}}(\Lambda) = d \rk_{A}(\Lambda) = dr\), and the \(A_{0}\)-module \(\Lambda\) may be written as \(\Lambda = \bigperp_{1 \leq i \leq dr} \Lambda^{(i)}\) with
one-dimensional \(\Lambda^{(i)}\). We have
\begin{equation}
	\Lambda_{N} = \{ \lambda \in \Lambda \mid \log \lambda \leq N \} = \{ \lambda \in \Lambda \mid \log_{0} \lambda \leq N/d\} = \bigperp \Lambda_{N/d}^{(i)},
\end{equation}
as the sum of \(A_{0}\)-modules is orthogonal. Since \(\Lambda^{(i)}\) is uqp-1, \(\Lambda\) is uqp-1 as \(A_{0}\)-module, and then uqp-\(d\) as an \(A\)-module.

Together with the preceding considerations, we find:

\begin{Proposition}
	Each \(A\)-lattice \(\Lambda\) of rank \(r\) in \(C_{\infty}\) is ultimately quasi-periodic with quasi-period \(d_{\infty}\).
\end{Proposition}

\subsection{} Now we come to the

\begin{proof}[Proof of Theorem \ref{Theorem.Distribution-on-frak-G-is-rational-function}]
\begin{enumerate}[wide, label=(\roman*)]
	\item Let \(N_{0} \geq \log u\) be a rational number such that \eqref{Eq.Spec-of-lattice-Lambda} is fulfilled for \(N \geq N_{0}\) and \(d = d_{\infty}\). Write
	\begin{equation}
		Z_{u, \Lambda}(S) = f_{u,N}(S^{1/n}) + g_{N}(S^{1/n})
	\end{equation}
	with
	\begin{equation}
		f_{u,N}(S^{1/n}) \defeq \sideset{}{^{\prime}} \sum_{\substack{\lambda \in K\Lambda \\ \lambda \equiv u~(\mathrm{mod } \Lambda) \\ \log \lambda \leq N}} S^{\log \lambda},
	\end{equation}
	a Laurent polynomial in \(S^{1/n}\), and the tail \(g_{N}(S^{1/n})\), which doesn't depend on \(u\). Here \(N \geq N_{0}\), and \(n\) is a common denominator for the 
	appearing \(\log \lambda\).
	
	Let \(\Spec_{(N, d_{\infty})}(\Lambda) = ( (x_{1},m_{1}), \dots, (x_{s}, m_{s}) )\). Due to \eqref{Eq.Spec-of-lattice-Lambda}, we have
	\begin{equation}
		\begin{split}
			g_{N}(S^{1/n}) 	&= (m_{1}S^{x_{1}} + \cdots + m_{s}S^{x_{s}})(1 + q_{\infty}^{r}S^{d_{\infty}} + q^{2r}S^{2d_{\infty}} + \cdots) \\
							&= (m_{1}S^{x_{1}} + \cdots + m_{s}S^{x_{s}})(1 - q_{\infty}^{r} S^{d_{\infty}})^{-1},
		\end{split}
	\end{equation}
	which gives (i).
	\item Assume \fbox{\(u \in \Lambda\)}. Then
	\[
		Z_{\Lambda}(1) = f_{0,N}(1) + g_{N}(1) = \#(\Lambda_{N} \smallsetminus \{0\}) + \#(\Lambda_{N + d_{\infty}} \smallsetminus \Lambda_{N})(1 - q_{\infty}^{r})^{-1} = {-}1,
	\]
	as \(\dim_{\mathds{F}} \Lambda_{N + d_{\infty}} = rd_{\infty} + \dim_{\mathds{F}} \Lambda_{N}\).
	
	Let now \fbox{\(u \notin \Lambda\)}. Since \(\lambda \mapsto u + \lambda\) is a bijection of \(\Lambda_{N}\) with the set 
	\(\{ \lambda' \in K\Lambda \mid \lambda' \equiv u \pmod{\Lambda} \text{ and } \log \lambda' \leq N \}\), the Laurent polynomial \(f_{u,N}\) contains exactly one term 
	more than the corresponding polynomial \(f_{0,N}\) for \(Z_{0, \Lambda} = Z_{\Lambda}\), viz, the term \(S^{\log u}\). Hence 
	\(Z_{u, \Lambda}(1) = Z_{0, \Lambda}(1) + 1 = 0\).
	\item We have \(\frac{\mathop{d}}{\mathop{dS}}(S^{\log \lambda})|_{S=1} = \log \lambda\), and the stated formula \eqref{Eq.Theorem-Difference-of-sums-of-logarithms}
	follows from \(Z_{u, \Lambda}(S) - Z_{0, \Lambda}(S) = f_{u,N}(S^{1/n}) - f_{0,N}(S^{1/n})\) for \(N \gg 0\).
	\item follows from \eqref{Eq.Homogenity-property-of-distribution-on-frak-G} and (ii).
\end{enumerate}	
\end{proof}

\begin{Remark}
	By \eqref{Eq.Distribution-on-frak-G} for each \(\boldsymbol{\omega} \in \Psi\), the rule 
	\(f \colon (\mathbf{u}, Y) \mapsto Z_{\mathbf{u}\boldsymbol{\omega}, Y_{\boldsymbol{\omega}}}'(1)\) yields a \(\mathds{Q}\)-valued distribution. The formula
	\eqref{Eq.Theorem-Difference-of-sums-of-logarithms} in Theorem \ref{Theorem.Distribution-on-frak-G-is-rational-function} may be stated as: The derivative \(g = Df\)
	of the distribution \(f\) is motivated by \(m \colon V \smallsetminus \{0\} \to \mathds{Q}\), \(m(\mathbf{x}) = \log \mathbf{x}\boldsymbol{\omega}\).	
\end{Remark}

\section{\(1\)-units}\label{Section.1-units}

Having the distributions \((Z_{*,*})\) and \((Z_{*,*}'(1))\) with values in characteristic zero, we next define a similar distribution with values in pro-\(p\)-groups.

\subsection{}\label{Subsection.Choices-for-1-units}%
We make the following choices:
\begin{itemize}
	\item \(\pi \in K\), a uniformizer at \(\infty\);
	\item for each \(n \in \mathds{N}\), an \(n\)-th root \(\pi^{1/n} \in C_{\infty}\) of \(\pi\), and such that for \(m \mid n\), the rule
	\((\pi^{1/n})^{n/m} = \pi^{1/m}\) is satisfied.
\end{itemize}
As with roots \(S^{1/n}\) of \(S\) in the last section, this requires Zorn's lemma in the general case, but as before, we will be working only on a finite level \(n\).

Then arbitrary powers \(\pi^{x}\) (\(x \in \mathds{Q}\)) are defined, with the usual rules holding. As the group \(\log(C_{\infty}^{*})\) equals \(\mathds{Q}\), we may 
split 
\begin{equation}
	\begin{split}
		C^{*}_{\infty}	&\overset{\cong}{\longrightarrow} \pi^{\mathds{Q}} \times U(C_{\infty}), \\
					z	&\longmapsto (\pi^{{-}\log z/d_{\infty}}, z\pi^{\log z/d_{\infty}})
	\end{split}
\end{equation}
where \(U(C_{\infty}) = \{ z \in C_{\infty} \mid \lvert z \rvert = 1\}\) is the unit sphere in \(C_{\infty}\). For each \(z \in U(C_{\infty})\), there exists a unique
root of unity \(\sgn(z)\) such that \(\sgn(z)^{-1}z\) is congruent to 1 modulo the valuation ideal of \(C_{\infty}\), i.e., such that \(\lvert \sgn(z)^{-1}z-1 \rvert < 1\).
This gives a surjective homomorphism
\begin{equation}
	\sgn \colon U(C_{\infty}) \longrightarrow \mu(C_{\infty})
\end{equation}
to the group \(\mu(C_{\infty})\) of roots of unity in \(C_{\infty}\). Let \(U^{(1)}(C_{\infty})\) be its kernel, the \textbf{group of} \(1\)\textbf{-units} of \(C_{\infty}\).
We extend \(\sgn\) to \(C_{\infty}^{*}\) by decreeing \(\sgn(\pi^{x}) = 1\) for each \(x \in \mathds{Q}\). Together, we find a decomposition
\begin{equation}\label{Eq.Decomposition-of-C-infty-star}\setcounter{subsubsection}{3}%
	\begin{split}
		C_{\infty}^{*}		&\overset{\cong}{\longrightarrow} \mathds{Q} \times \mu(C_{\infty}) \times U^{(1)}(C_{\infty}) \\
				z			&\longmapsto ({-}\log z/d_{\infty}, \sgn(z), \langle z \rangle),
	\end{split}
\end{equation}
where \(\langle z \rangle \defeq z \cdot \pi^{\log z/d_{\infty}} \sgn(z)^{-1}\) is the \(1\)-unit part of \(z\). Note that the component \({-}\log z/d_{\infty}\) is
intrinsic, while \(\sgn(z)\) and \(\langle z \rangle\) depend on the choice of \(\pi\) and its roots. We also observe
\subsubsection{} The decomposition \eqref{Eq.Decomposition-of-C-infty-star} imposes the discrete topology to the factors \(\mathds{Q}\) and \(\mu(C_{\infty})\).

\subsection{}\label{Subsection.Fixation-of-A-lattice-Gamma-of-rank-r}%
In what follows, we fix an \(A\)-lattice \(\Lambda \subset C_{\infty}\) of rank \(r\), and put for rational numbers \(N < N'\)
\begin{equation}
	\Lambda_{N,N'} = \Lambda_{N'} \smallsetminus \Lambda_{N}.
\end{equation}
For some \(a \in A\) of degree \(d > 0\), choose a section \(s\) of the \(\mathds{F}\)-linear map \(\Lambda \to \Lambda/a\Lambda\), and let \(R \defeq s(\Lambda/a\Lambda)\).
We assume that \(N \gg_{\Lambda} 0\) and also 
\begin{equation}
	N \geq C_{0} \defeq \max_{x \in R} \log x.	
\end{equation}
Then
\begin{equation}\label{Eq.Bijection-Lambda-N-N+d-times-R}
	\begin{split}
		\Lambda_{N,N+d} \times R		&\overset{\cong}{\longrightarrow} \Lambda_{N+d,N+2d} \\
			(\lambda, x)				&\longmapsto a\lambda + x
	\end{split}
\end{equation}
is bijective. Due to our assumptions,
\begin{equation}\label{Eq.Identities-for-logarithm-and-sign}
	\log(a\lambda + x) = \log a + \log \lambda \quad \text{and} \quad \sgn(a\lambda + x) = \sgn(a)\sgn(\lambda).
\end{equation}
Let 
\begin{equation}
	H(X) \defeq \prod_{x \in R} (X-x) = X^{q^{rd}} + h_{rd-1}X^{q^{rd-1}} + \cdots + h_{0}X
\end{equation}
be the \(\mathds{F}\)-linear polynomial corresponding to \(R\). Its coefficients, being composed of elementary symmetric functions in the \(x\), satisfy
\begin{equation}
	\log h_{i} \leq (q^{rd} - q^{i})C_{0} \quad (0 \leq i < rd).
\end{equation}
For \(\lambda \in \Lambda_{N,N+d}\) we find with easy (and omitted) estimates:
\begin{equation}
	H(a\lambda)/(a\lambda)^{q^{rd}} = 1 + O(q^{{-}N})
\end{equation}
(that is, there exists \(C_{1} > 0\) such that the error term is \(\leq C_{1}q^{{-}N}\).) Therefore, by \eqref{Eq.Identities-for-logarithm-and-sign},
\begin{equation}\label{Eq.Product-of-langle-alambda-x-rangle}
	\prod_{x \in R} \langle a\lambda + x \rangle = \sgn(a\lambda)^{{-}q^{rd}} \pi^{q^{rd}\log(a\lambda)/d_{\infty}} H(a\lambda) = \langle a\lambda \rangle^{q^{rd}}(1 + O(q^{{-}N})).
\end{equation}
Put 
\begin{equation}
	U_{\Lambda, N} \defeq \sideset{}{^{\prime}} \prod_{\lambda \in \Lambda_{N}} \langle \lambda \rangle, \qquad U_{\Lambda, N, N'} \defeq \prod_{\lambda \in \Lambda_{N,N'}} \langle \lambda \rangle, \quad \text{and} \quad U_{u, \Lambda, N} \defeq \sideset{}{^{\prime}} \prod_{\substack{\lambda \in K\Lambda \\ \lambda \equiv u~(\mathrm{mod } \Lambda) \\ \log \lambda \leq N}} \langle \lambda \rangle.
\end{equation}
Taking the product of \eqref{Eq.Product-of-langle-alambda-x-rangle} over \(\lambda \in \Lambda_{N,N+d}\) yields 
\begin{equation}
	U_{\Lambda, N+d, N+2d} = \langle a \rangle^{q^{rd}\#(\Lambda_{N,N+d})}(U_{\Lambda, N, N+d})^{q^{rd}}(1 + O(q^{{-}N})).
\end{equation}
Now, as \(\langle a \rangle\) and \(U_{\Lambda, N}\) are \(1\)-units, we see that 
\begin{equation}\label{Eq.Lim-U-Lambda-N+Kd-N+(k+1)d-equals-1}
	\lim_{k \to \infty} U_{\Lambda, N+kd, N+(k+1)d} = 1
\end{equation}
and thus \(\lim_{k \to \infty} U_{\Lambda, N+kd}\) exists. One easily shows that neither the convergence nor the value of the limit depends on the shape \(N+kd\) of the
sequence approaching \(\infty\). Thus finally
\begin{equation}
	U_{\Lambda} \defeq \lim_{N \to \infty} U_{\Lambda, N} \in U^{(1)}(C_{\infty})
\end{equation}
is well-defined. With the same reasoning, we get for each \(u \in K\Lambda\) the existence of the limit
\begin{equation}
	U_{u, \Lambda} \defeq \lim_{N \to \infty} U_{u,\Lambda, N}
\end{equation}
as an element of \(U^{(1)}(C_{\infty})\).

\begin{Proposition}\label{Proposition.Existence-of-limit}
	\begin{enumerate}[label=\(\mathrm{(\roman*)}\)]
		\item For each \(u \in K\Lambda\), the limit \(U_{u,\Lambda} = \lim_{N \to \infty} U_{u, \Lambda, N}\) exists.
		\item \(U_{u,\Lambda}\) depends only on the class of \(u\) modulo \(\Lambda\).
		\item For a lattice pair \(\Lambda \subset \Lambda'\subset K\Lambda\) and \(v \in K\Lambda\), the distribution relation
		\[
			\prod_{\substack{u \in K\Lambda \\ u \equiv v~(\mathrm{mod } \Lambda')}} U_{u, \Lambda} = U_{v, \Lambda'} \quad \text{holds.}
		\]
		\item If \(c \in C_{\infty}^{*}\) then 
		\[
			U_{cu, c\Lambda} = \begin{cases} U_{u, \Lambda}	&(u \notin \Lambda) \\ \langle c \rangle^{-1}U_{u,\Lambda}	&(u \in \Lambda). \end{cases}
		\]
	\end{enumerate}	
\end{Proposition}

\begin{proof}
	(i) has been shown, and (ii) and (iii) are then easy consequences. As to (iv): We have \(U_{cu, c\Lambda} = \lim_{N \to \infty} \langle c \rangle^{i(N)} U_{u, \Lambda}\)
	with \(i(N) = \#(\Lambda_{N})\) in case \(u \notin \Lambda\) and \(i(N) = \#(\Lambda_{N} \smallsetminus \{0\})\) in case \(u \in \Lambda\). Now
	\(\langle c \rangle^{\#(\Lambda_{N})}\) converges to 1, as \(\#(\Lambda_{N})\) is a power of \(q\) growing with \(N\).
\end{proof}

The proposition assures that for given \(\boldsymbol{\omega} \in \Psi\), the rule 
\(f \colon (\mathbf{u}, Y) \mapsto U_{\mathbf{u}\boldsymbol{\omega}, Y_{\boldsymbol{\omega}}}\) is a distribution on \(\mathfrak{Y}\) with values in \(U^{(1)}(C_{\infty})\).
It is motivated by \(m \colon V \smallsetminus \{0\} \to U^{(1)}(C_{\infty})\), \(m(\mathbf{x}) = \langle \mathbf{x} \boldsymbol{\omega} \rangle\).

\section{Roots of unity}\label{Section.Roots-of-unity}

We would like to dispose of a similar distribution with values in \(\mu(C_{\infty})\) motivated by \(\sgn(\lambda)\), where 
\(0 \neq \lambda \in \Lambda = Y_{\boldsymbol{\omega}}\). This fails however, due to the lacking of an analogue of \eqref{Eq.Lim-U-Lambda-N+Kd-N+(k+1)d-equals-1}
for the sign function \(\sgn\).

\subsection{} For each subset \(S\) of \(C_{\infty}\), define
\begin{equation}
	\sgn(S) \defeq \text{the subgroup of \(\mu(C_{\infty})\) generated by \(\sgn(x)\), where \(0 \neq x \in S\)}.
\end{equation}
If \(S\) is finite, then put
\begin{equation}
	\varepsilon(S) \defeq \sideset{}{^{\prime}} \prod_{x \in S} \sgn(x).
\end{equation}
Fix an \(A\)-lattice \(\Lambda \subset C_{\infty}\) of rank \(r\). We first observe:

\begin{Proposition}
	\(\sgn(\Lambda)\) is finite.
\end{Proposition}

\begin{proof}
	Let \(a \in A\) of degree \(d > 0\) and \(N \gg 0\) such that \eqref{Eq.Bijection-Lambda-N-N+d-times-R} is fulfilled. Proceeding as in 
	\ref{Subsection.Fixation-of-A-lattice-Gamma-of-rank-r}, we find that \(\sgn(\Delta_{N+2d})\) is generated by \(\sgn(\Lambda_{N+d})\) and \(\sgn(A) = \mu_{w}\),
	the group of \(w\)-th root of unity, \(w = q_{\infty} - 1\). Now use induction.
\end{proof}

\subsection{} In what follows, we assume that \(N \in \mathds{Q}\) is large enough with respect to the lattices \(\Lambda \subset \Lambda'\), the choices of \(a \in A\)
of degree \(d > 0\), and of the system \(R\) of representatives for \(\Lambda/a\Lambda\), and w.r.t. \(\log u\), so that, e.g., \eqref{Eq.Bijection-Lambda-N-N+d-times-R}
and \eqref{Eq.Identities-for-logarithm-and-sign} apply.

\begin{Theorem}\label{Theorem.On-roots-of-unity}
	\begin{enumerate}[label=\(\mathrm{(\roman*)}\)]
		\item The root of unity
		\begin{equation}\label{Eq.Theorem-independence-of-root-of-unity}
			\varepsilon^{\Lambda} \defeq \varepsilon(\Lambda_{N})^{q_{\infty}^{r}}/\varepsilon(\Lambda_{N+d_{\infty}})
		\end{equation}
		is independent of \(N \gg 0\) and therefore an invariant of \(\Lambda\). For a multiple \(d\) of \(d_{\infty}\),
		\begin{equation}\label{Eq.Theorem-identitiy-for-roots-of-unity}
			\varepsilon(\Lambda_{N})^{q^{rd}}/\varepsilon(\Lambda_{N+d}) = (\varepsilon^{\Lambda})^{(q^{rd}-1)/(q_{\infty}^{r}-1)}
		\end{equation}
		holds. Further, for \(0 \neq c \in K\),
		\begin{equation}\label{Eq.Theorem-scalar-invariance-of-root-of-unity}
			\varepsilon^{c\Lambda} = \varepsilon^{\Lambda}.
		\end{equation}
		\item Define for \(u \in K\Lambda \smallsetminus \Lambda\)
		\begin{equation}
			\varepsilon_{u}^{\Lambda} \defeq \varepsilon_{u,N}^{\Lambda} \defeq \prod_{\substack{\lambda \in K\Lambda \\ \lambda \equiv u~(\mathrm{mod } \Lambda) \\ \log \lambda \leq N}} \sgn(\lambda)/\varepsilon(\Lambda_{N}).
		\end{equation}
		This is independent of \(N \gg 0\), and an invariant of the class of \(u \pmod{\Lambda}\) and \(\Lambda\). For \(0 \neq c \in K\), 
		\begin{equation}\label{Eq.Theorem-sign-dependence-by-index-root-of-unity}
			\varepsilon_{cu}^{c\Lambda} = \sgn(c) \varepsilon_{u}^{\Lambda}.
		\end{equation}
		\item Let \(\Lambda \subset \Lambda' \subset K\Lambda\) with another lattice \(\Lambda'\) of rank \(r\), and define
		\begin{equation}\label{Eq.Theorem-definition-root-for-lattice-pair}
			\varepsilon^{\Lambda'|\Lambda} \defeq \Big( \sideset{}{^{\prime}} \prod_{u \in \Lambda'/\Lambda} \varepsilon_{u}^{\Lambda} \Big)^{-1} = \varepsilon(\Lambda_{N})^{[\Lambda':\Lambda]}( \varepsilon(\Lambda_{N}') )^{-1},
		\end{equation}
		where \(N\) is large enough. For \(v \in K\Lambda \smallsetminus \Lambda'\), we have
		\begin{equation}\label{Eq.Theorem-product-formula-for-roots}
			\prod_{\substack{u \in K\Lambda/\Lambda \\ u \equiv v~(\mathrm{mod } \Lambda')}} \varepsilon_{u}^{\Lambda} \cdot \varepsilon^{\Lambda'|\Lambda} = \varepsilon_{v}^{\Lambda'}.
		\end{equation}
		\item Let \(w_{r} \defeq q_{\infty}^{r} - 1\). Then
		\begin{equation}\label{Eq.Theorem-index-formula-for-root-of-lattice-pair}
			(\varepsilon^{\Lambda'|\Lambda})^{w_{r}} = (\varepsilon^{\Lambda})^{[\Lambda' : \Lambda]}/\varepsilon^{\Lambda'}.
		\end{equation}
	\end{enumerate}	
\end{Theorem}

\begin{proof}
	\begin{enumerate}[wide, label=(\roman*)]
		\item Choose \(a \in A\) of degree \(d > 0\). From \eqref{Eq.Bijection-Lambda-N-N+d-times-R} and \eqref{Eq.Identities-for-logarithm-and-sign} we find
		\begin{equation}
			\varepsilon(\Lambda_{N+d, N+2d}) = \varepsilon(\Lambda_{N,N+d})^{q^{rd}}.
		\end{equation}
		This implies that
		\[
			\varepsilon^{(d)} \defeq \varepsilon(\Lambda_{N})^{q^{rd}}/\varepsilon(\Lambda_{N+d})
		\]
		is independent of \(N \gg 0\). Calculation shows that for \(d_{i} = \deg a_{i}\) (\(i=1,2\)),
		\[
			\varepsilon^{(d_{1}+d_{2})} = \varepsilon^{(d_{2})}(\varepsilon^{(d_{1})})^{q^{rd_{2}}} = \varepsilon^{(d_{1})}(\varepsilon^{(d_{2})})^{q^{rd_{1}}},
		\]
		and so
		\begin{equation}
			(\varepsilon^{(d_{1})})^{(q^{rd_{2}}-1)} = (\varepsilon^{(d_{2})})^{q^{rd_{1}}-1}.
		\end{equation}
		Now we choose \(a_{1}, a_{2}\) such that \(\gcd(d_{1}, d_{2}) = d_{\infty}\) and apply \ref{Subsection.Collection-of-facts} to get 
		\(\varepsilon^{(d_{1})} = (\varepsilon^{(d_{\infty})})^{(q^{rd_{1}}-1)/(q_{\infty}^{r}-1)}\) where \(\varepsilon^{(d_{\infty})} = \varepsilon^{\Lambda}\) is an 
		invariant of \(\Lambda\). This shows \eqref{Eq.Theorem-independence-of-root-of-unity} and \eqref{Eq.Theorem-identitiy-for-roots-of-unity}, and
		\eqref{Eq.Theorem-scalar-invariance-of-root-of-unity} is obvious, as \(\sgn(c)\) is a \(w\)-th root of unity (\(w = q_{\infty}-1\)) and 
		\(q_{\infty}^{r}-1 \equiv 0 \pmod{w}\).
		\item Replacing the bound \(N\) by \(N + d_{\infty}\), the newly appearing elements of the index sets of the numerator and of the denominator correspond to each
		other by \(\lambda \leftrightarrow \lambda - u\). Since \(N > \log u\), we have \(\sgn(\lambda) = \sgn(\lambda - u)\), that is 
		\(\varepsilon_{u, N+d_{\infty}}^{\Lambda} = \varepsilon_{u,N}^{\Lambda}\). This gives the independence on \(N\); the independence of the representative
		\(u \pmod{\Lambda}\) is obvious, and \eqref{Eq.Theorem-sign-dependence-by-index-root-of-unity} comes out as the number of factors in the numerator of
		\(\varepsilon_{u,N}^{\Lambda}\) is one larger than the number of factors in the denominator \(\varepsilon(\Lambda_{N})\).
		\item By definition of \(\varepsilon_{u}^{\Lambda}\),
		\[
			\sideset{}{^{\prime}} \prod_{u \in \Lambda'/\Lambda} \varepsilon_{u}^{\Lambda} = \varepsilon(\Lambda_{N}' \smallsetminus \Lambda_{N}) \varepsilon(\Lambda_{N})^{1-[\Lambda':\Lambda]}
		\]
		if \(N > \max \log u\), where \(u\) runs through a system of representatives for \(\Lambda'/\Lambda \smallsetminus \{0\}\), and the right hand side equals
		\(\varepsilon(\Lambda_{N}') \cdot \varepsilon(\Lambda_{N})^{{-}[\Lambda':\Lambda]}\). This gives the equality stated in 
		\eqref{Eq.Theorem-definition-root-for-lattice-pair}, while \eqref{Eq.Theorem-product-formula-for-roots} results from a straightforward calculation left
		to the reader.
		\item Inserting \eqref{Eq.Theorem-definition-root-for-lattice-pair}, the wanted identity is equivalent with (\(\ell \defeq [\Lambda':\Lambda]\)):
		\[
			\frac{\varepsilon(\Lambda_{N})^{\ell w_{r}}}{\varepsilon(\Lambda'_{N})^{w_{r}}} = \frac{\varepsilon(\Lambda_{N})^{q_{\infty}^{r} \cdot \ell} \cdot \varepsilon(\Lambda'_{N + d_{\infty}})}{\varepsilon(\Lambda_{N+d_{\infty}})^{\ell} \cdot \varepsilon(\Lambda_{N}')^{q_{\infty}^{r}}}
		\]
		for \(N \gg 0\), and by cancelling, with
		\[
			\frac{\varepsilon(\Lambda_{N})^{\ell}}{\varepsilon(\Lambda_{N}')} = \frac{\varepsilon(\Lambda_{N+d_{\infty}})^{\ell}}{\varepsilon(\Lambda'_{N + d_{\infty}})}.
		\]
		But the latter are two different expressions for the same quantity \(\varepsilon^{\Lambda'|\Lambda}\).
	\end{enumerate}
\end{proof}

\begin{Remarks}
	\begin{enumerate}[label=(\roman*), wide]
		\item \eqref{Eq.Theorem-product-formula-for-roots} together with (ii) of the theorem states that, for fixed \(\boldsymbol{\omega} \in \Psi\), the map
		\(g \colon (\mathbf{u}, Y) \mapsto \varepsilon_{\mathbf{u}\boldsymbol{\omega}}^{Y_{\boldsymbol{\omega}}}\) is a derived distribution with values in 
		\(\mu(C_{\infty})\), and is motivated by \(m \colon V \smallsetminus \{0\} \to \mu(C_{\infty})\), where \(m(\mathbf{x}) = \sgn(\mathbf{x}\boldsymbol{\omega})\).
		And \eqref{Eq.Theorem-index-formula-for-root-of-lattice-pair} asserts that its \(w_{r}\)-th power \(G \colon (\mathbf{u}, Y) \mapsto (\varepsilon_{\mathbf{u}\boldsymbol{\omega}}^{Y_{\boldsymbol{\omega}}})^{w_{r}}\) is the derivative of the distribution \(F\), where \(F(0,Y) = \varepsilon^{Y}\) and
		\(F(\mathbf{u}, Y) = G(\mathbf{u},Y)F(0,Y)\) if \(\mathbf{u} \notin Y\).
		\item One is tempted to ask if already \(g\) itself could be completed to a distribution like \(G = g^{w_{r}}\), or if possibly a proper divisor \(n\) of 
		\(w_{r}\) does the job. This is actually the case for \(n = w_{r}/(q-1)\), as the discriminant \(\Delta\) defined in Corollary 
		\ref{Corollary.Canonical-discriminant-form}, regarded as a modular form for \(\Gamma_{Y} = \GL(Y)\), has a \((q-1)\)-th root \(h_{Y}\) as a holomorphic 
		function on the Drinfeld space \(\Omega = C_{\infty}^{*} \backslash \Psi\) (\cite{BassonBreuerPink-tA}, \cite{gekeler2023drinfeld}). These \(h_{Y}\) are 
		well-defined up to \((q-1)\)-th roots of unity, and may be chosen in a manner consistent for all \(Y\). We will not pursue that topic here. 	
	\end{enumerate}	
\end{Remarks}

\section{Product expansions for discriminants and division points}\label{Section.Product-expansions}

\subsection{} We fix a lattice pair \(Y'|Y\) in \(V = K^{r}\) and let \(\Delta^{Y'|Y}\) be the associated discriminant function as in \eqref{Eq.Discriminant-of-Phi}, that is
\begin{equation}\label{Eq.Form-of-lattice-pair-in-terms-of-d-bold-u}
	\Delta^{Y'|Y}(\boldsymbol{\omega}) = \Big( \sideset{}{^{\prime}} \prod_{\mathbf{u} \in Y'/Y} d_{\mathbf{u}}^{Y}(\boldsymbol{\omega}) \Big)^{-1}.
\end{equation}
We also fix \(\boldsymbol{\omega} \in \Psi\) and put \(\Lambda \defeq Y_{\boldsymbol{\omega}}\), \(\Lambda' = Y_{\boldsymbol{\omega}}'\), and 
\(\Delta^{\Lambda'|\Lambda} = \Delta^{Y'|Y}(\boldsymbol{\omega})\). The following results will be intrinsic for \(\Lambda\) and \(\Lambda'\) and independent from the
presentation \(\Lambda = Y_{\boldsymbol{\omega}}\), etc., which is only used to embed our current situation into the context of Sections \ref{Section.Introduction-and-Notation}
and \ref{Section.Distributions-and-Derived-Distributions}.

From \eqref{Eq.Form-of-lattice-pair-in-terms-of-d-bold-u} and the definition of 
\(d_{\mathbf{u}}^{Y}(\boldsymbol{\omega}) = e^{Y_{\boldsymbol{\omega}}}(\mathbf{u}\boldsymbol{\omega})\), we get
\begin{equation}
	(\Delta^{Y'|Y})^{-1} = \sideset{}{^{\prime}} \prod_{u \in \Lambda'/\Lambda} e^{\Lambda}(u).
\end{equation}
Therefore, we first treat \(e^{\Lambda}(u)\).

\subsection{} We have
\[
	e^{\Lambda}(u) = \lim_{N \to \infty} e^{\Lambda_{N}}(u)
\]	
with
\begin{equation}
	e^{\Lambda_{N}}(u) = u \sideset{}{^{\prime}} \prod_{\lambda \in \Lambda_{N}} (1 - u/\lambda) = u \sideset{}{^{\prime}} \prod_{\lambda \in \Lambda_{N}} \left( \frac{\lambda + u}{\lambda} \right) = \prod_{\substack{\lambda \in K\Lambda \\ \lambda \equiv u~(\mathrm{mod } \Lambda) \\ \log \lambda \leq N }} \lambda / \sideset{}{^{\prime}} \prod_{\lambda \in \Lambda_{N}} \lambda,	
\end{equation}
supposing that \(N \gg 0\). Now we decompose the factors according to \eqref{Eq.Decomposition-of-C-infty-star}. This yields
\begin{equation}
	e^{\Lambda_{N}}(u) = \pi^{{-}k} \varepsilon_{u,N}^{\Lambda} U_{u,\Lambda, N} / U_{\Lambda, N}^{-1},
\end{equation}
where for \(N \geq \log u\),
\[
	k = d_{\infty}^{-1} \Big( \sum_{\substack{\lambda \in K\Lambda \\ \lambda \equiv u ~(\mathrm{mod } \Lambda) \\ \log \lambda \leq N}} \log \lambda - \sideset{}{^{\prime}} \sum_{\lambda \in \Lambda_{N}} \log \lambda \Big).
\]
For \(N \to \infty\), \(k\) and \(\varepsilon_{u,N}^{\Lambda}\) become stationary, that is, for \(N \gg 0\),
\begin{equation}
	k = d_{\infty}^{{-}1}[Z_{u, \Lambda}'(1) - Z_{\Lambda}'(1)]
\end{equation}
by Theorem \ref{Theorem.Distribution-on-frak-G-is-rational-function}(iii) and
\begin{equation}
	\varepsilon_{u,N}^{\Lambda} = \varepsilon_{u}^{\Lambda}
\end{equation}
as defined in Theorem \ref{Theorem.On-roots-of-unity}, while \(U_{u,\Lambda,N} \to U_{u, \Lambda}\) and \(U_{\Lambda, N} \to U_{\Lambda}\).

\subsection{} Now we use the distribution properties of \(Z_{*,*}\), \(U_{*,*}\) and \(\varepsilon_{*}^{*}\) to get 
\begin{equation}
	\sideset{}{^{\prime}} \sum_{u \in \Lambda'/\Lambda} (Z_{u, \Lambda}'(1) - Z_{\Lambda}'(1) ) = Z_{\Lambda'}'(1) - [\Lambda': \Lambda]Z_{\Lambda}'(1), \footnotemark
\end{equation}
\footnotetext{The reader will have noticed that the primes in \(\sideset{}{^{\prime}} \sum\), \(\Lambda'\) and \(Z_{*,*}'\) all have different meanings.}
\begin{equation}
	\sideset{}{^{\prime}} \prod_{u \in \Lambda'/\Lambda} (U_{u, \Lambda} U_{\Lambda}^{-1}) = U_{\Lambda'} \cdot U_{\Lambda}^{{-}[\Lambda':\Lambda]}
\end{equation}
and
\begin{equation}
	\sideset{}{^{\prime}}\prod_{u \in \Lambda'/\Lambda} \varepsilon_{u}^{\Lambda} = (\varepsilon^{\Lambda'|\Lambda})^{-1}.
\end{equation}
This finally gives our principal result (part (iii) of the Main Theorem).

\begin{Theorem} \label{Theorem.Main-Theorem}
	Let \(\Lambda'|\Lambda\) be a lattice pair of rank \(r\) in \(C_{\infty}\) and \(u \in K\Lambda \smallsetminus \Lambda\).
	\begin{enumerate}[label=\(\mathrm{(\roman*)}\)]
		\item The value of the division point \(e^{\Lambda}(u)\) of the Drinfeld module \(\phi^{\Lambda}\) associated with \(\Lambda\) is
		\begin{equation}\label{Eq.Theorem-formula-exponential-in-terms-of-roots-of-unity}
			e^{\Lambda}(u) = \pi^{{-}k} \varepsilon_{u}^{\Lambda} U_{u, \Lambda} U_{\Lambda}^{-1}
		\end{equation}
		with \(k = d_{\infty}^{-1}(Z_{u, \Lambda}'(1) - Z_{\Lambda}'(1))\).
		\item The value \(\Delta(\Lambda'|\Lambda)\) of the discriminant (see \ref{Subsection.Lattice-rank}) is given by
		\begin{equation}
			\Delta(\Lambda'|\Lambda) = \pi^{k} \varepsilon^{\Lambda'|\Lambda} U_{\Lambda}^{[\Lambda':\Lambda]} U_{\Lambda'}^{-1}
		\end{equation}
		with \(k = d_{\infty}^{-1}(Z_{\Lambda}'(1) - [\Lambda':\Lambda]Z_{\Lambda}'(1))\).
	\end{enumerate}	
\end{Theorem}

\begin{Corollary}
	Let \(a \in A\) have degree \(d > 0\). Then the usual discriminant \(\Delta_{a}^{Y}(\boldsymbol{\omega})\) of the Drinfeld module \(\phi^{Y_{\boldsymbol{\omega}}}\)
	is
	\begin{equation}\label{Eq.Formula-for-discriminant-of-Drinfeld-module-phi-Y-boldomega}
		\Delta_{a}(\Lambda) = \Delta_{a}^{Y}(\boldsymbol{\omega}) = \pi^{k} \sgn(a) (\varepsilon^{\Lambda})^{(q^{rd}-1)/w_{r}} U_{\Lambda}^{q^{rd}-1}
	\end{equation}	
	with \(k = d_{\infty}^{-1}(1 - q^{rd})Z_{\Lambda}'(1)\) and \(w_{r} = q_{\infty}^{r} - 1\).
\end{Corollary}

\begin{proof}
	This is the case of Theorem \ref{Theorem.Main-Theorem} where \(Y' = a^{-1}Y\). Here 
	\(\Delta_{a}^{Y}(\boldsymbol{\omega}) = a\Delta^{Y'|Y}(\boldsymbol{\omega}) = a\Delta(\Lambda'|\Lambda)\), where \(\Delta(\Lambda'|\Lambda)\) has the following components:
	\begin{itemize}
		\item \(\pi^{k}\), \(k = d_{\infty}^{-1}(Z_{a^{-1}\Lambda}'(1) - q^{rd}Z_{\Lambda}'(1))\)
		\item \(\varepsilon^{a^{-1}\Lambda|\Lambda}\)
		\item \(U_{\Lambda}^{q^{rd}}/U_{a^{-1}\Lambda}\),
	\end{itemize}
	and \(a\) splits as \(a = \pi^{{-}\log a/d_{\infty}} \sgn(a) \langle a \rangle\). 
	
	Now \(Z_{a^{-1}\Lambda}'(1) = Z_{\Lambda}'(1) + \log a\) by \eqref{Eq.Homogenity-property-of-distribution-on-frak-G}, the term \(\log a\) cancels, and \(k\) in 
	\eqref{Eq.Formula-for-discriminant-of-Drinfeld-module-phi-Y-boldomega} is as stated.
	
	Next, \(\varepsilon^{a^{-1}\Lambda | \Lambda} = \varepsilon(\Lambda_{N})^{q^{rd}} \varepsilon^{{-}1}( (a^{{-}1}\Lambda)_{N} )\) for \(N \gg 0\) by
	\eqref{Eq.Theorem-definition-root-for-lattice-pair}. Further, \((a^{-1}\Lambda)_{N} = a^{-1}(\Lambda_{N+d})\), so
	\[
		\varepsilon( (a^{-1}\Lambda)_{N}) = \sgn(a)^{1-\#(\Lambda_{N+d})}\varepsilon(\Lambda_{N+d}) = \varepsilon(\Lambda_{N+d}),
	\]
	as \(\#(\Lambda_{N+d})-1\) is divisible by \(w = q_{\infty} -1\), and
	\[
		\varepsilon^{a^{-1}\Lambda|\Lambda} = \varepsilon(\Lambda_{N})^{q^{rd}}/\varepsilon(\Lambda_{N+d}) = (\varepsilon^{\Lambda})^{(q^{rd}-1)/w_{r}}
	\]
	by \eqref{Eq.Theorem-identitiy-for-roots-of-unity}. This gives the unit root part of \eqref{Eq.Formula-for-discriminant-of-Drinfeld-module-phi-Y-boldomega}. Finally,
	\(U_{\Lambda}^{q^{rd}}/U_{a^{-1}\Lambda} = \langle a \rangle^{-1} U_{\Lambda}^{q^{rd}-1}\) by \ref{Proposition.Existence-of-limit}(iv).
\end{proof}

\begin{Corollary-Definition}\label{Corollary.Canonical-discriminant-form}
	There is a \textbf{canonical discriminant} \(\Delta(\Lambda)\) with the property
	\begin{equation}\label{Eq.Definition-canonical-discriminant-property}
		\Delta_{a}(\Lambda) = \sgn(a) \Delta(\Lambda)^{(q^{rd}-1)/w_{r}}
	\end{equation}
	for \(0 \neq a \in A\) of degree \(d\). It also satisfies the product formula
	\begin{equation}\label{Eq.Product-formula-for-canonical-discriminant-property}
		\Delta(\Lambda) = \pi^{k}\varepsilon^{\Lambda}U_{\Lambda}^{w_{r}}
	\end{equation}
	with \(k = {-}d_{\infty}^{{-}1} w_{r} Z_{\Lambda}'(1)\), and
	\begin{equation}\label{Eq.Weight-of-Delta-as-modular-form}
		\text{the weight of \(\Delta\) as a modular form is \(w_{r} = q_{\infty}^{r} -1\).}
	\end{equation}
	(This is item(i) of the Main Theorem.)
\end{Corollary-Definition}

\begin{proof}
	Let \(a,a'\) be elements of \(A\) of respective degrees \(d, d'\) and with \(\gcd(d,d') = d_{\infty}\). By \eqref{Eq.gcd-condition} applied to \(q' = q^{r}\),
	we also have \(\gcd(q^{rd}-1, q^{rd'}-1) = q_{\infty}^{r} - 1 = w_{r}\). Now we apply \eqref{Eq.Existence-of-b-infinity} to the modular forms 
	\(\tilde{\Delta}_{a} \defeq \sgn(a)^{-1} \Delta_{a}\), \(\tilde{\Delta}_{a'} \defeq \sgn(a')^{-1} \Delta_{a'}\), of weights \(q^{rd}-1\), \(q^{rd'}-1\), respectively, and,
	writing \(d_{\infty} = xd + x'd'\) (\(x,x' \in \mathds{Z}\)), we define
	\begin{equation} \label{Eq.Definition-Delta-no-extras}
		\Delta \defeq \tilde{\Delta}_{a}^{x} \tilde{\Delta}_{a'}^{x'}.
	\end{equation}
	Then \eqref{Eq.Product-formula-for-canonical-discriminant-property} follows from \eqref{Eq.Formula-for-discriminant-of-Drinfeld-module-phi-Y-boldomega}, which also shows
	the independence of \(\Delta\) from the choices of \(a\), \(a'\), \(x\), and \(x'\), as well as \eqref{Eq.Definition-canonical-discriminant-property}. Finally,
	\eqref{Eq.Definition-Delta-no-extras} gives the weight \(w_{r}\) for \(\Delta\), i.e., \eqref{Eq.Weight-of-Delta-as-modular-form}.
\end{proof}

\begin{Remark}
	In \cite{gekeler2023drinfeld} we had defined \(\Delta\) of weight \(w_{r}\), but it was well-defined only up to roots of unity. The same problem occurs 
	in \cite{BassonBreuerPink-tA} Proposition 16.4, where \(\Delta\) was also defined up to roots of unity. The present \(\Delta\) is in so far canonical as it depends 
	only on the sign function \enquote{\(\sgn\)} restricted to \(K\) (or, what amounts to the same, the choice of the uniformizer \(\pi\) of \(K_{\infty}\) 
	modulo \(\mathfrak{p}_{\infty}^{2}\), where \(\mathfrak{p}_{\infty}\) is the valuation ideal in \(K_{\infty}\)), but not on the other choices made.	
\end{Remark}

\begin{Corollary}\label{Corollary.Absolute-values-of-discriminants}
	The absolute values of the discriminants are as follows: 
	\begin{enumerate}[label=\(\mathrm{(\roman*)}\)]
		\item \(\lvert \Delta(\Lambda'|\Lambda) \rvert = q^{k}\), \(k = [\Lambda':\Lambda] Z_{\Lambda}'(1) - Z_{\Lambda'}'(1)\)
		\item \(\lvert \Delta_{a}(\Lambda) \rvert = q^{k}\), \(k = (q^{rd}-1)Z_{\Lambda}'(1)\)
		\item \(\lvert \Delta(\Lambda) \rvert = q^{k}\), \(k = w_{r}Z_{\Lambda}'(1)\)
	\end{enumerate}
\end{Corollary}

\subsection{} So far, we had \(\Delta(\Lambda'|\Lambda) = \Delta^{Y'|Y}(\boldsymbol{\omega})\), \(\Delta_{a}(\Lambda) = \Delta_{a}^{Y}(\boldsymbol{\omega})\),
\(\Delta(\Lambda) = \Delta^{Y}(\boldsymbol{\omega})\) with a fixed \(\boldsymbol{\omega} \in \Psi\). Now we allow \(\boldsymbol{\omega}\) to vary over \(\Psi\).
For \(H\) one of the factors \(\mathds{Q}\), \(\mu(C_{\infty})\), \(U^{(1)}(C_{\infty})\) of \(C_{\infty}^{*}\) (see \eqref{Eq.Decomposition-of-C-infty-star}), we let
\(\Maps(\Psi, H)\) be the group of \(H\)-valued functions on \(\Psi\). The following hold since we know the corresponding properties argumentwise:

\subsubsection{} The map
\begin{align*}
	\mathfrak{Y}^{*} 		&\longrightarrow \Maps(\Psi, \mathds{Q}) \\
		(\mathbf{u}, Y)		&\longmapsto \log d_{\mathbf{u}}^{Y}(\boldsymbol{\omega}) = Z_{\mathbf{u}\boldsymbol{\omega}, Y_{\boldsymbol{\omega}}}'(1) - Z_{Y_{\boldsymbol{\omega}}}'(1)	
	\intertext{is a derived distribution, motivated by}
	V \smallsetminus \{0\}	&\longrightarrow \Maps(\Psi, \mathds{Q}). \\
					x		&\longmapsto \log \mathbf{x}\boldsymbol{\omega}	
	\intertext{It is the derivative of the distribution}
		\mathfrak{Y}		&\longrightarrow \Maps(\Psi, \mathds{Q}) \\
		(\mathbf{u}, Y)		&\longmapsto Z_{\mathbf{u}\boldsymbol{\omega}, Y_{\boldsymbol{\omega}}}'(1).
\end{align*}
\subsubsection{} Second, 
\begin{align*}
	\mathfrak{Y}			&\longrightarrow \Maps(\Psi, U^{(1)}(C_{\infty})) \\
			(\mathbf{u}, Y)	&\longmapsto U_{\mathbf{u},Y}
	\intertext{is a distribution, and is motivated by \(\mathbf{x} \mapsto \langle \mathbf{x}\boldsymbol{\omega}\rangle\).
	\subsubsection{}Third,}
	\mathfrak{Y}^{*}		&\longrightarrow \Maps(\Psi, \mu(C_{\infty})) \\
	(\mathbf{u}, Y)			&\longmapsto \varepsilon_{\mathbf{u}\boldsymbol{\omega}}^{Y_{\boldsymbol{\omega}}}		
\end{align*}
is a derived distribution motivated by \(\mathbf{x} \mapsto \sgn(\mathbf{x}\boldsymbol{\omega})\). Its \(w_{r}\)-th power 
\((\mathbf{u}, Y) \mapsto (\varepsilon_{\mathbf{u}\boldsymbol{\omega}}^{Y_{\boldsymbol{\omega}}})^{w_{r}}\) is the derivative of a distribution.

We put these together to the next result (item (ii) of the Main Theorem).

\begin{Corollary}
	Let \(\mathcal{O}(\Psi)^{*}\) be the group of invertible holomorphic functions on the Drinfeld space \(\Psi = \Psi^{r}\). Then the map
	\begin{align*}
		F \colon \mathfrak{Y}	&\longrightarrow \mathcal{O}(\Psi)^{*} \\
			(\mathbf{u}, Y)		&\longmapsto \begin{cases} (d_{\mathbf{u}}^{Y})^{w_{r}} \Delta^{Y}	&(\mathbf{u} \notin Y) \\ \Delta^{Y}	&(\mathbf{u} \in Y) \end{cases}
	\end{align*}
	is a distribution with derivative \(G \colon (\mathbf{u}, Y) \mapsto (d_{\mathbf{u}}^{Y})^{w_{r}}\). \(G\) is motivated by 
	\(m \colon V \smallsetminus \{0\} \to \mathcal{O}(\Psi)^{*}\), \(\mathbf{x} \mapsto (\mathbf{x}\boldsymbol{\omega})^{w_{r}}\).
\end{Corollary}

\subsection{A heuristic consideration} We would like to regard the product formula \eqref{Eq.Product-formula-for-canonical-discriminant-property} for \(\Delta\) (and similarly,
the formula \eqref{Eq.Theorem-formula-exponential-in-terms-of-roots-of-unity} for \(e^{\Lambda}(u)\)) in Euler's style as a formula
\begin{equation} \label{Eq.Heuristic-consideration-Delta-of-Lambda}
	? \Delta(\Lambda) = \sideset{}{^{\prime}} \prod_{\lambda \in \Lambda} \lambda^{w_{r}} ?
\end{equation}
and accordingly for \(e^{\Lambda}(u)\). A first attempt to endow it with some reason was to replace it by
\begin{equation} \label{Eq.Heuristic-consideration-Delta-of-Lambda-as-limit}
	? \Delta(\Lambda) = \lim_{N \to \infty} \sideset{}{^{\prime}} \prod_{\lambda \in \Lambda_{N}} \lambda^{w_{r}} ?,
\end{equation}
and to split the latter into components according to \eqref{Eq.Decomposition-of-C-infty-star}. This works perfectly for the \fbox{\(1\)-unit part}, in view of Proposition
\ref{Proposition.Existence-of-limit}; that is, the \(1\)-unit part of \eqref{Eq.Heuristic-consideration-Delta-of-Lambda-as-limit} is true. It works less smoothly for the
\fbox{absolute value part}, since \(\lim_{N \to \infty} \sideset{}{^{\prime}} \sum_{\lambda \in \Lambda_{N}} \log \lambda\) doesn't exist. But at least, 
\[
	\lim_{N \to \infty} \Big( \sideset{}{^{\prime}} \sum_{\substack{\lambda \in K\Lambda \\ \lambda \equiv u ~(\mathrm{mod } \Lambda) \\ \log \lambda \leq N}} \log \lambda - \sideset{}{^{\prime}} \sum_{\lambda \in \Lambda_{N}} \log \lambda \Big)
\]
exists and equals \(d_{\infty}^{-1}(Z_{u, \Lambda}'(1) - Z_{\Lambda}'(1))\) by \eqref{Eq.Theorem-Difference-of-sums-of-logarithms}. Then \(\lim_{N \to \infty} \sideset{}{^{\prime}} \sum_{\lambda \in \Lambda_{N}} \log \lambda\) is assigned the value \(Z_{\Lambda}'(1)\) through the distribution property of \(Z_{*,*}'(1)\). Note that
this evaluation resembles the procedure of analytic continuation. Finally, for the \fbox{unit root part}, we have to slightly alter 
\[
	\sideset{}{^{\prime}} \prod_{\lambda \in \Lambda_{N}} \sgn(\lambda)^{w_{r}} = \Big( \sideset{}{^{\prime}} \prod_{\lambda \in \Lambda_{N}} \sgn(\lambda)^{q_{\infty}^{r}} \Big)/ \sideset{}{^{\prime}} \prod_{\lambda \in \Lambda_{N}} \sgn(\lambda)
\]
to
\[
	\Big( \sideset{}{^{\prime}} \prod_{\lambda \in \Lambda_{N}} \sgn(\lambda)^{q_{\infty}^{r}} \Big)/ \sideset{}{^{\prime}} \prod_{\lambda \in \Lambda_{N+d_{\infty}}} \sgn(\lambda) \qquad (N \gg 0)
\]
in order to get the correct value \(\varepsilon^{\Lambda}\) for the unit root part of formula \eqref{Eq.Product-formula-for-canonical-discriminant-property}. Together, 
it seems justified to regard \eqref{Eq.Product-formula-for-canonical-discriminant-property} as a regularized product like \eqref{Eq.Heuristic-consideration-Delta-of-Lambda}
over \(\Lambda\).

\section{The case \(A = \mathds{F}[T]\)} \label{Section.The-case-A-IF-T}

As an example for the preceding, we now restrict to the most important case where the curve \(\mathcal{C}\) in \ref{Subsection.First-introduction-of-notation} is the
projective line and \(\infty\) the usual place at infinity, and so the Drinfeld ring \(A\) is a polynomial ring \(\mathds{F}[T]\) as in 
\ref{Subsubsection.Simple-example-C-projective-line-infty-place-at-infinity}.

Here, the corresponding distributions for \fbox{\(r=1\)} (related to the Carlitz module and its division points) have already been studied 1974 by Hayes \cite{Hayes74} and
1980 by Galovich and Rosen \cite{GalovichRosen81}. In the case \fbox{\(r=2\)} there is a large amount of work about the corresponding modular forms, starting 1980 with
\cite{Goss80} and \cite{Gekeler80}, \cite{Gekeler85}, \cite{Gekeler84}, \cite{Gekeler88}, and since then continued by many authors. The higher rank case \fbox{\(r > 2\)} 
has become a topic of serious research only about 2017 with a series of preprints \cite{basson2018drinfeldI}, \cite{basson2018drinfeldII}, \cite{basson2018drinfeldIII} 
by Basson, Breuer and Pink (which now are to appear in \cite{BassonBreuerPink-tA}) and the ongoing series \enquote{On Drinfeld modular forms of higher rank I, II, \dots} 
by the present author \cite{Gekeler17}, \cite{Gekeler22}, \cite{gekeler2023drinfeld}. Actually it was the desire to find a common framework for the 
distributive aspects of these (and other) papers that led to the present work.

\subsection{} Assume \fbox{\(r=1\)}. As each lattice \(Y\) in \(V = K = \mathds{F}(T)\) is free of rank \(1\) of shape \(Ay\), we may replace the distribution domain
\(\mathfrak{Y}\) by \(\mathfrak{Y} = K/A\) and the distribution property \eqref{Eq.Evaluation-of-f-for-sublattice-Y'-and-Y} by 
\begin{equation}
	\sum_{\substack{u \in K/A \\ u \equiv v~(\mathrm{mod}~n^{-1}A)}} f(u) = f(v) \qquad \text{for } 0 \neq n \in A.
\end{equation} 
This is the point of view of \cite{GalovichRosen81}. We take \(\pi = T^{-1}\) as uniformizer at infinity, \(\sgn(a) =\) leading coefficient of \(a \in A\) as a polynomial in \(T\). The
\(Z\)-function is
\begin{equation}
	Z_{A}(S) = (q-1) \sum_{i \geq 0} (qS)^{i} = (q-1)/(1-qS)^{-1}.
\end{equation}
For \(0 \neq u \in K/A\), represented by \(a/n\) with \((a,n) = 1\), \(a, n \in A\) with \(n\) monic and \(0 \leq d \defeq \deg a < \deg n\), we have
\begin{equation}
	Z_{u,A}(S) = S^{d-\deg n} + Z_{A}(S),
\end{equation}
and so
\begin{equation}
	Z_{A}'(1) = q/(q-1), \quad Z_{u,A}'(1) = q/(q-1) + d - \deg n.
\end{equation}
Inserting these into the formulas of the last section, we get expressions for the sizes of \(d_{u}^{A}(1) = e_{A}(u)\) and \(\Delta^{A}(1) = \Delta(A)\).
Let \(\overline{\pi}\) be the period of the Carlitz module (well-defined up to a \((q-1)\)-th root of unity), so that \(A_{\overline{\pi}} = A\overline{\pi}\) is
its period lattice.\footnote{\(\overline{\pi}\) should not be confused with \(\pi = \) uniformizer at \(\infty = T^{-1}\) (here).} By the general formalism
\begin{equation}
	d_{u}^{A}(\overline{\pi}) = \overline{\pi} d_{u}^{A}(1)
\end{equation}
is the corresponding division point of the Carlitz module and 
\begin{equation}
	\Delta^{A}(\overline{\pi}) = \overline{\pi}^{1-q} \Delta^{A}(1) = 1
\end{equation}
by definition of the Carlitz module. We find from \eqref{Eq.Formula-for-discriminant-of-Drinfeld-module-phi-Y-boldomega}
\begin{equation}
	\overline{\pi}^{q-1} = \Delta^{A}(1) = \pi^{k} \varepsilon^{A} U_{A}^{q-1},
\end{equation}
where \(k = (1-q) Z_{A}'(1) = {-}q\), \(\varepsilon^{A} = {-}1\) and 
\(U_{A} = \lim_{N \to \infty} \sideset{}{^{\prime}} \prod_{\substack{a \in A \\ \deg a < N}} \langle a \rangle\), \(\langle a \rangle = \sgn(a)^{{-}1}T^{{-}\deg a}d\),
which agrees with the formulas in \cite{Gekeler86} IV 4.10 or \cite{Gekeler88} 4.11. Similarly, for the division points,
\begin{equation}
	\begin{split}
	\log d_{u}^{A}(\overline{\pi}) 	&= \log \overline{\pi} + \log d_{u}^{A}(1) \\	
									&= \frac{q}{q-1} + ( Z_{u,A}'(1) - Z_{A}'(1)) = Z_{u,A}'(1) = \frac{q}{q-1} + d - \deg n,
	\end{split}
\end{equation}
in accordance with the values given, e.g., in \cite{GalovichRosen81} or \cite{Gekeler86} IV 4.13.

\subsection{} Now we consider the case where \fbox{\(r = 2\)}. We restrict to demonstrate how \(\lvert \Delta(\boldsymbol{\omega})\rvert\) and 
\(\lvert d_{\mathbf{u}}^{Y}(\boldsymbol{\omega})\rvert\) or rather their logarithms for \(Y = A^{2}\) may be computed from our general results.
For this we assume that \(\boldsymbol{\omega} \in \Psi^{2}\) actually belongs to \(\Omega^{2} = \{ (\omega_{1}, \omega_{2}) \in \Psi^{2} \mid \omega_{2} = 1\}\), i.e.,
that \(\boldsymbol{\omega} = (\omega, 1)\) with \(\omega \in C_{\infty} \smallsetminus K_{\infty}\). This is the usual framework of Drinfeld modular forms
of rank 2 for the group \(\Gamma = \GL(2,A)\) as described, e.g., in \cite{Gekeler88}. Moreover, we assume that \(\omega\) belongs to the fundamental domain \(\mathbf{F}\) for \(\Gamma\),
\begin{equation}
	\mathbf{F} = \{ \omega \in C_{\infty} \smallsetminus K_{\infty} \mid \lvert \omega \rvert = \lvert \omega \rvert_{i} \geq 1\},
\end{equation}
where \(\lvert \omega_{i} \rvert_{i} \defeq \inf_{x \in K_{\infty}} \lvert \omega -x \rvert\). Now the logarithms of the invertible functions \(\Delta(\omega)\)
and \(d_{\mathbf{u}}^{Y}(\omega)\) depend only on \(\log \omega\) and interpolate linearly from integer values of \(\log \omega\). (This has been known for long time 
by results of van der Put, and has been generalized to higher ranks in \cite{Gekeler22} Theorems 2.4 and 2.6.) Therefore we may assume that
\begin{equation}\label{Eq.Definition-of-bold-F-ell}
	\omega \in \mathbf{F}_{\ell} \defeq \{ \omega \in C_{\infty} \smallsetminus K_{\infty} \mid \lvert \omega \rvert = \lvert \omega \rvert_{i} = q^{\ell} \}
\end{equation}
with \(\ell \in \mathds{N}_{0} = \{0,1,2,\dots\}\). Let \(\Lambda\) be the lattice \(Y_{\boldsymbol{\omega}} = A\omega + A \subset C_{\infty}\). Due to
\eqref{Eq.Definition-of-bold-F-ell},
\begin{equation}
	Z_{\Lambda}(S) = (q-1) \frac{1 - (qS)^{\ell}}{1 - qS} + (q^{2} - 1) \frac{(qS)^{\ell}}{1 - q^{2}S}.
\end{equation}
Let \(0 \neq \mathbf{u} = n^{-1}(a_{1}, a_{2}) \in n^{-1}Y\), \(a_{1}, a_{2}, n \in A\), \(d_{1} = \deg a_{1}\), \(d_{2} = \deg a_{2}\), \(d_{1}, d_{2} < \deg n\),
\(u \defeq \mathbf{u}\boldsymbol{\omega} = n^{-1}(a_{1}\omega + a_{2})\). Then
\begin{equation}\label{Eq.Formulas-for-log-u}
	\log u = \begin{cases} \ell + d_{1} - \deg n,	&\text{if } d_{2} \leq d_{1} + \ell, \\ d_{2} - \deg n,	&\text{if } d_{2} > d_{1} + \ell \end{cases}
\end{equation}
(note that always \(\log u < \ell)\) and therefore
\begin{equation}
	Z_{u, \Lambda}(S) = \begin{cases} Z_{\Lambda}(S) + (qS)^{\log u} - (q-1) \sum_{0 \leq i < \log u} (qS)^{i},	&\text{if } \log u \geq 0, \\ Z_{\Lambda}(S) + S^{\log u},	&\text{if } \log u < 0. \end{cases}
\end{equation}
We find
\begin{equation}
	Z_{\Lambda}'(1) = \frac{q^{2} + q - q^{\ell + 1}}{q^{2} - 1}
\end{equation}
and
\[
	Z_{u, \Lambda}'(1) - Z_{\Lambda}'(1)	 = \frac{q}{q-1} (q^{\log u} - 1) = \begin{cases} \frac{q}{q-1}( q^{\ell + d_{1} - \deg n} - 1),	&\text{if } \log u \geq 0, \\ \log u,	&\text{if } \log u < 0. \end{cases}
\]
Finally, from Theorem \ref{Theorem.Main-Theorem} and its corollaries,
\begin{align}
	\log \Delta(\omega)				&= q^{2} + q - q^{\ell + 1} \label{Eq.Formula-log-Delta-omega}
	\intertext{and}
	\log d_{\mathbf{u}}^{Y}(\omega)	&= \begin{cases} \frac{q}{q-1}(q^{\ell + d_{1} - \deg n} - 1),	&\text{if } \log u \geq 0, \\ \log u,	&\text{if } \log u < 0. \end{cases} \label{Eq.Formula-log-d-boldu-Y-of-omega}
\end{align}
for \(\omega \in \mathbf{F}_{\ell}\), \(u = n^{.1}(a_{1}\omega + a_{2})\).

The first of these, \eqref{Eq.Formula-log-Delta-omega}, has been found in \cite{Gekeler97} Theorem 2.13, using a difficult argument involving the van der Put transform of \(\Delta\)
and its Fourier coefficients. To the author's best knowledge, \eqref{Eq.Formula-log-d-boldu-Y-of-omega}, although known to him for quite some time, is so far nowhere
published.

If \(a_{1} = 0\), then always \(\log u = d_{2} - \deg n < 0\), and \(\log d_{\mathbf{u}}^{Y}(\omega) = d_{2} - \deg n\) is constant along \(\mathbf{F}_{0}\), 
\(\mathbf{F}_{1}\), \(\mathbf{F}_{2}\), \dots If \(a_{1} \neq 0\) then for \(\ell \geq d_{2} - d_{1}\), the first case of \eqref{Eq.Formulas-for-log-u} prevails,
and if moreover \(\ell \geq \deg n - d_{1}\), then \(\log d_{\mathbf{u}}^{Y}(\omega)\) is given by the first formula of \eqref{Eq.Formula-log-d-boldu-Y-of-omega},
and grows very fast with \(\ell\).

\subsection{} We conclude this set of examples with the case \fbox{\(r=3\)}, where we restrict to give formulas for \(\log \Delta(\boldsymbol{\omega})\). The situation,
although considerably more complex, is analogous with the just considered case \(r = 2\). Again we assume that \(\boldsymbol{\omega} = (\omega_{1}, \omega_{2}, \omega_{3})\)
lies in the fundamental domain \(\mathbf{F} = \mathbf{F}^{3}\) of \(\Gamma = \GL(3,A)\). That is
\begin{equation}
	\omega_{1}, \omega_{2} \text{ and } \omega_{3} = 1 \text{ are orthogonal (see \ref{Subsubsection.Being.uqp-d-is-stable}) and } \lvert \omega_{1} \rvert \geq \lvert \omega_{2} \rvert \geq 1.
\end{equation}
Moreover, still because of \cite{Gekeler22} Theorems 2.4 and 2.6, we assume that \(\boldsymbol{\omega}\) lies above a vertex of the Bruhat-Tits building, which here means that
\[
	a \defeq \log \omega_{1} \quad \text{and} \quad b \defeq \log \omega_{2}
\]
are integers. For such \(\boldsymbol{\omega}\), let \(\Lambda = Y_{\boldsymbol{\omega}} = A \omega_{1} \oplus A\omega_{2} + A \subset C_{\infty}\). We distinguish the
cases 
\begin{equation}
	(1)~a=b=0; \qquad (2)~ a>b=0; \qquad (3)~a=b>0; \qquad (4)~a>b>0.
\end{equation}
\textbf{Case} (1) \fbox{\(a = b = 0\)} Here the \(Z\)-function is \(Z_{\Lambda}(S) = (q^{3}-1)/(1-q^{3}S)\), which gives
\begin{equation}
	\log \Delta(\boldsymbol{\omega}) = q^{3}.
\end{equation}
This is well-known and could be seen by bare eye.

\textbf{Case} (2) \fbox{\(a > b = 0\)} Here for \(N \in \mathds{N}_{0}\),
\[
	\dim_{\mathds{F}} \Lambda_{N} = 2(N+1) \text{ if } N < a \quad \text{and} \quad 3(N+1) - a \text{ if } N \geq a.
\]
This implies
\[
	Z_{\Lambda}(S) = (q^{2} - 1) \frac{(q^{2}S)^{a} - 1}{q^{2}S-1} + (q^{3}-1) \frac{(q^{2}S)^{a}}{1 - q^{3}S}
\]
and we find
\begin{equation}
	\log \Delta(\boldsymbol{\omega}) = (q^{3}-1) Z_{\Lambda}'(1) = q^{2a+3} - (q^{3}-1)q^{2} \frac{q^{2a}-1}{q^{2}-1}.
\end{equation}

\textbf{Case} (3) \fbox{\(a = b > 0\)} We have for \(N \in \mathds{N}_{0}\)
\[
	\dim_{\mathds{F}} \Lambda_{N} = N+1 \text{ if } N<a \quad \text{and} \quad 3(N+1) - 2a \text{ if } N \geq a,
\]
and so
\[
	Z_{\Lambda}(S) = (q-1) \frac{(qS)^{a} - 1}{qS -1} + (q^{3}-1) \frac{(qS)^{a}}{1-q^{3}S},
\]
which gives
\begin{equation}
	\log \Delta(\boldsymbol{\omega}) = (q^{3}-1) Z_{\Lambda}'(1) = q^{3} - (q^{a}-1)(q^{2}+q).
\end{equation}

\textbf{Case} (4) \fbox{\(a>b>0\)} The dimension of \(\Lambda_{N}\) is given by \(\dim \Lambda_{N} = N+1\), \(2(N+1) - b\), \(3(N+1) - a-b\) if \(N < b\),
\(b \leq N < a\), \(N \geq a\), respectively. We obtain
\[
	Z_{\Lambda}(S) = (q-1) \frac{(qS)^{b}-1}{qS-1} + (q^{2}-1)(qS)^{b} \frac{(qS)^{a-b}-1}{q^{2}S-1} + (q^{3}-1)q^{a-b} \frac{(qS)^{a}}{1-q^{3}S}
\]
and thus
\begin{equation}
	\log \Delta(\boldsymbol{\omega}) = (q^{3}-1) Z_{\Lambda}'(1)
\end{equation}
with the above \(Z_{\Lambda}(S)\) (we omit to write it out).

\subsection{} A few of the values of \(\log \Delta(\boldsymbol{\omega})\) for \(\boldsymbol{\omega} \in \mathbf{F} = \mathbf{F}^{3}\) had already been
calculated, with great pain, in \cite{Gekeler17} (combine Figure 3 with Theorem 4.13 loc. cit.) in form of its increments on the Bruhat-Tits building. Luckily, these
agree with the values obtained above!

It is an exercise in intelligent notation to write down a general formula for \(\log \Delta(\boldsymbol{\omega})\), where \(\boldsymbol{\omega} \in \mathbf{F}^{r}\) lies
above a vertex in the fundamental domain of the Bruhat-Tits building \(\mathcal{BT}^{r}\), or to write a computer program. Similarly, it isn't but a matter of patience
to work out \(Z_{u, \Lambda}(S)\) for \(u = \mathbf{u}\boldsymbol{\omega}\) and \(\mathbf{u} \in K^{r}/A^{r}\), where \(r \geq 2\) is general, and thereby to get closed
expressions for \(\log d_{\mathbf{u}}^{A^{r}}(\boldsymbol{\omega})\) in the style of \eqref{Eq.Formula-log-d-boldu-Y-of-omega}.

\printbibliography

\end{document}